\documentclass[reqno,11pt]{amsart}
\usepackage{amsmath}
\usepackage{amsxtra}
\usepackage{amscd}
\usepackage{amsthm}
\usepackage{amsfonts}
\usepackage{amssymb}
\usepackage{eucal}
\usepackage{scalerel}
\usepackage[matrix,arrow,curve]{xy}

\theoremstyle{plain}
\newtheorem{Thm}[subsection]{Theorem}
\newtheorem{Cor}[subsection]{Corollary}
\newtheorem{Lem}[subsection]{Lemma}
\newtheorem{Prop}[subsection]{Proposition}
\newtheorem{Conj}[subsection]{Conjecture}

\theoremstyle{definition}
\newtheorem{Def}[subsection]{Definition}

\theoremstyle{remark}

\newtheorem{Rem}[subsection]{Remark}

\newtheorem{conj}{Conjecture}

\numberwithin{equation}{section}
\renewcommand{\rm}{\normalshape}

\newif\ifShowLabels
\ShowLabelstrue
\newdimen\theight
\def\TeXref#1{%
    \leavevmode\vadjust{\setbox0=\hbox{{\tt
        \quad\quad  {\small \rm #1}}}%
    \theight=\ht0
    \advance\theight by \lineskip
    \kern -\theight \vbox to
    \theight{\rightline{\rlap{\box0}}%
    \vss}%
    }}%

\ShowLabelsfalse

\renewcommand{\sec}[2]{\section{#2}\label{S:#1}%
    \ifShowLabels \TeXref{{S:#1}} \fi}
\newcommand{\ssec}[2]{\subsection{#2}\label{SS:#1}%
    \ifShowLabels \TeXref{{SS:#1}} \fi}

\newcommand{\refs}[1]{Section ~\ref{S:#1}}
\newcommand{\refss}[1]{Section ~\ref{SS:#1}}

\newcommand{\reft}[1]{Theorem ~\ref{T:#1}}
\newcommand{\refl}[1]{Lemma ~\ref{L:#1}}
\newcommand{\refp}[1]{Proposition ~\ref{P:#1}}
\newcommand{\refc}[1]{Corollary ~\ref{C:#1}}
\newcommand{\refd}[1]{Definition ~\ref{D:#1}}
\newcommand{\refr}[1]{Remark ~\ref{R:#1}}
\newcommand{\refe}[1]{\eqref{E:#1}}
\newcommand{\refco}[1]{Conjecture ~\ref{Co:#1}}

\newenvironment{thm}[1]%
    { \begin{Thm} \label{T:#1}  \ifShowLabels \TeXref{T:#1} \fi }%
    { \end{Thm} }

\renewcommand{\th}[1]{\begin{thm}{#1} \sl }
\renewcommand{\eth}{\end{thm} }

\newenvironment{lemma}[1]%
    { \begin{Lem} \label{L:#1}  \ifShowLabels \TeXref{L:#1} \fi }%
    { \end{Lem} }
\newcommand{\lem}[1]{\begin{lemma}{#1} \sl}
\newcommand{\elem}{\end{lemma}}

\newenvironment{propos}[1]%
    { \begin{Prop} \label{P:#1}  \ifShowLabels \TeXref{P:#1} \fi }%
    { \end{Prop} }
\newcommand{\prop}[1]{\begin{propos}{#1}\sl }
\newcommand{\eprop}{\end{propos}}

\newenvironment{corol}[1]%
    { \begin{Cor} \label{C:#1}  \ifShowLabels \TeXref{C:#1} \fi }%
    { \end{Cor} }
\newcommand{\cor}[1]{\begin{corol}{#1} \sl }
\newcommand{\ecor}{\end{corol}}

\newenvironment{defeni}[1]%
    { \begin{Def} \label{D:#1}  \ifShowLabels \TeXref{D:#1} \fi }%
    { \end{Def} }
\newcommand{\defe}[1]{\begin{defeni}{#1} \sl }
\newcommand{\edefe}{\end{defeni}}

\newenvironment{remark}[1]%
    { \begin{Rem} \label{R:#1}  \ifShowLabels \TeXref{R:#1} \fi }%
    { \end{Rem} }
\newcommand{\rem}[1]{\begin{remark}{#1}}
\newcommand{\erem}{\end{remark}}

\newenvironment{conjec}[1]%
    { \begin{Conj} \label{Co:#1}  \ifShowLabels \TeXref{Co:#1} \fi }%
    { \end{Conj} }
\renewcommand{\conj}[1]{\begin{conjec}{#1} \sl }
\newcommand{\econj}{\end{conjec}}

\newcommand{\eq}[1]%
    { \ifShowLabels \TeXref{E:#1} \fi
       \begin{equation} \label{E:#1} }
\newcommand{\eeq}{ \end{equation} }

\newcommand{\prf}{ \begin{proof} }
\newcommand{\epr}{ \end{proof} }

\usepackage[dvipsnames]{xcolor}

\newcommand\nc{\newcommand}

\nc{\HC}{{\mathcal{HC}}}
\nc{\on}{\operatorname}
\nc{\BA}{{\mathbb{A}}}
\nc{\BC}{{\mathbb{C}}}
\nc{\BD}{{\mathbb{D}}}
\nc{\BF}{{\mathbb{F}}}
\nc{\BG}{{\mathbb{G}}}
\nc{\BM}{{\mathbb{M}}}
\nc{\BN}{{\mathbb{N}}}
\nc{\BO}{{\mathbb{O}}}
\nc{\BQ}{{\mathbb{Q}}}
\nc{\BP}{{\mathbb{P}}}
\nc{\BR}{{\mathbb{R}}}
\nc{\BZ}{{\mathbb{Z}}}
\nc{\BS}{{\mathbb{S}}}

\nc{\CA}{{\mathcal{A}}}
\nc{\CB}{{\mathcal{B}}}
\nc{\CalC}{{\mathcal C}}
\nc{\CalD}{{\mathcal D}}
\nc{\CE}{{\mathcal{E}}}
\nc{\CF}{{\mathcal{F}}}
\nc{\CG}{{\mathcal{G}}}
\nc{\CH}{{\mathcal{H}}}
\nc{\CK}{{\mathcal{K}}}
\nc{\CL}{{\mathcal{L}}}
\nc{\CM}{{\mathcal{M}}}
\nc{\CMM}{{\mathcal{M}^{\operatorname{gen}}_\hbar(-\rho)}}
\nc{\CN}{{\mathcal{N}}}
\nc{\CO}{{\mathcal{O}}}
\nc{\CP}{{\mathcal{P}}}
\nc{\CQ}{{\mathcal{Q}}}
\nc{\CR}{{\mathcal{R}}}
\nc{\CS}{{\mathcal{S}}}
\nc{\CT}{{\mathcal{T}}}
\nc{\CU}{{\mathcal{U}}}
\nc{\CV}{{\mathcal{V}}}
\nc{\CW}{{\mathcal{W}}}
\nc{\CX}{{\mathcal{X}}}
\nc{\CY}{{\mathcal{Y}}}
\nc{\CZ}{{\mathcal{Z}}}

\nc{\gen}{{\operatorname{gen}}}
\nc{\cM}{{\check{\mathcal M}}{}}
\nc{\csM}{{\check{\mathcal A}}{}}
\nc{\obM}{{\overset{\circ}{\mathbf M}}{}}
\nc{\oCA}{{\overset{\circ}{\mathcal A}}{}}
\nc{\obA}{{\overset{\circ}{\mathbf A}}{}}
\nc{\ooM}{{\overset{\circ}{M}}{}}
\nc{\osM}{{\overset{\circ}{\mathsf M}}{}}
\nc{\vM}{{\overset{\bullet}{\mathcal M}}{}}
\nc{\nM}{{\underset{\bullet}{\mathcal M}}{}}
\nc{\obD}{{\overset{\circ}{\mathbf D}}{}}
\nc{\cp}{{\overset{\circ}{\mathbf p}}{}}
\nc{\ofZ}{{\overset{\circ}{\mathfrak Z}}{}}

\nc{\fa}{{\mathfrak{a}}}
\nc{\fb}{{\mathfrak{b}}}
\nc{\fg}{{\mathfrak{g}}}
\nc{\fgl}{{\mathfrak{gl}}}
\nc{\fh}{{\mathfrak{h}}}
\nc{\fj}{{\mathfrak{j}}}
\nc{\fl}{{\mathfrak{l}}}
\nc{\fm}{{\mathfrak{m}}}
\nc{\fn}{{\mathfrak{n}}}
\nc{\fu}{{\mathfrak{u}}}
\nc{\fp}{{\mathfrak{p}}}
\nc{\frr}{{\mathfrak{r}}}
\nc{\fs}{{\mathfrak{s}}}
\nc{\ft}{{\mathfrak{t}}}
\nc{\fw}{{\mathfrak{w}}}
\nc{\fz}{{\mathfrak{z}}}
\nc{\ofT}{{\overline{\mathfrak T}}}
\nc{\ofS}{{\overline{\mathfrak S}}}
\nc{\fsl}{{\mathfrak{sl}}}
\nc{\hsl}{{\widehat{\mathfrak{sl}}}}
\nc{\hgl}{{\widehat{\mathfrak{gl}}}}
\nc{\hg}{{\widehat{\mathfrak{g}}}}
\nc{\chg}{{\widehat{\mathfrak{g}}}{}^\vee}
\nc{\hn}{{\widehat{\mathfrak{n}}}}
\nc{\chn}{{\widehat{\mathfrak{n}}}{}^\vee}

\nc{\fA}{{\mathfrak{A}}}
\nc{\fB}{{\mathfrak{B}}}
\nc{\fD}{{\mathfrak{D}}}
\nc{\fE}{{\mathfrak{E}}}
\nc{\fF}{{\mathfrak{F}}}
\nc{\fG}{{\mathfrak{G}}}
\nc{\fI}{{\mathfrak{I}}}
\nc{\fJ}{{\mathfrak{J}}}
\nc{\fK}{{\mathfrak{K}}}
\nc{\fL}{{\mathfrak{L}}}
\nc{\fM}{{\mathfrak{M}}}
\nc{\fN}{{\mathfrak{N}}}
\nc{\frP}{{\mathfrak{P}}}
\nc{\fQ}{{\mathfrak Q}}
\nc{\fR}{{\mathfrak R}}
\nc{\fS}{{\mathfrak S}}
\nc{\fT}{{\mathfrak{T}}}
\nc{\fU}{{\mathfrak{U}}}
\nc{\fW}{{\mathfrak{W}}}
\nc{\fZ}{{\mathfrak{Z}}}

\nc{\ba}{{\mathbf{a}}}
\nc{\bb}{{\mathbf{b}}}
\nc{\bc}{{\mathbf{c}}}
\nc{\bd}{{\mathbf{d}}}
\nc{\be}{{\mathbf{e}}}
\nc{\bi}{{\mathbf{i}}}
\nc{\bj}{{\mathbf{j}}}
\nc{\bn}{{\mathbf{n}}}
\nc{\bp}{{\mathbf{p}}}
\nc{\bq}{{\mathbf{q}}}
\nc{\bu}{{\mathbf{u}}}
\nc{\bv}{{\mathbf{v}}}
\nc{\bx}{{\mathbf{x}}}
\nc{\by}{{\mathbf{y}}}
\nc{\bw}{{\mathbf{w}}}
\nc{\bA}{{\mathbf{A}}}
\nc{\bB}{{\mathbf{B}}}
\nc{\bC}{{\mathbf{C}}}
\nc{\bD}{{\mathbf{D}}}
\nc{\bE}{{\mathbf{E}}}
\nc{\bK}{{\mathbf{K}}}
\nc{\bH}{{\mathbf{H}}}
\nc{\bM}{{\mathbf{M}}}
\nc{\bN}{{\mathbf{N}}}
\nc{\bO}{{\mathbf{O}}}
\nc{\bQ}{{\mathbf Q}}
\nc{\bS}{{\mathbf{S}}}
\nc{\bT}{{\mathbf{T}}}
\nc{\bV}{{\mathbf{V}}}
\nc{\bW}{{\mathbf{W}}}
\nc{\bX}{{\mathbf{X}}}
\nc{\bP}{{\mathbf{P}}}
\nc{\bZ}{{\mathbf{Z}}}
\nc{\bY}{{\mathbf{Y}}}

\nc{\sA}{{\mathsf{A}}}
\nc{\sB}{{\mathsf{B}}}
\nc{\sC}{{\mathsf{C}}}
\nc{\sD}{{\mathsf{D}}}
\nc{\sF}{{\mathsf{F}}}
\nc{\sK}{{\mathsf{K}}}
\nc{\sM}{{\mathsf{M}}}
\nc{\sO}{{\mathsf{O}}}
\nc{\sP}{{\mathsf{P}}}
\nc{\sQ}{{\mathsf{Q}}}
\nc{\sR}{{\mathsf{R}}}
\nc{\sV}{{\mathsf{V}}}
\nc{\sW}{{\mathsf{W}}}
\nc{\sZ}{{\mathsf{Z}}}
\nc{\sfp}{{\mathsf{p}}}
\nc{\sr}{{\mathsf{r}}}
\nc{\st}{{\mathsf{t}}}
\nc{\sfu}{{\mathsf{u}}}
\nc{\sfb}{{\mathsf{b}}}
\nc{\sfc}{{\mathsf{c}}}
\nc{\sd}{{\mathsf{d}}}
\nc{\sg}{{\mathsf{g}}}
\nc{\sk}{{\mathsf{k}}}
\nc{\sfl}{{\mathsf{l}}}

\nc{\BK}{{\bar{K}}}

\nc{\tA}{{\widetilde{\mathbf{A}}}}
\nc{\tB}{{\widetilde{\mathcal{B}}}}
\nc{\tg}{{\widetilde{\mathfrak{g}}}}
\nc{\tG}{{\widetilde{G}}}
\nc{\TM}{{\widetilde{\mathbb{M}}}{}}
\nc{\tN}{{\widetilde{\mathcal{N}}}{}}
\nc{\tO}{{\widetilde{\mathsf{O}}}{}}
\nc{\tU}{{\widetilde{\mathfrak{U}}}{}}
\nc{\TZ}{{\tilde{Z}}}
\nc{\tZ}{\widetilde{Z}{}}
\nc{\tx}{{\tilde{x}}}
\nc{\tbv}{{\tilde{\bv}}}
\nc{\tfP}{{\widetilde{\mathfrak{P}}}{}}
\nc{\tz}{{\tilde{\zeta}}}
\nc{\tmu}{{\tilde{\mu}}}

\nc{\td}{\ddot{\underline{d}}{}}
\nc{\tzeta}{\widetilde{\zeta}{}}
\nc{\hd}{{\widehat{\underline{d}}}}
\nc{\hG}{{\widehat{G}}}
\nc{\hBP}{\widehat{\mathbb P}{}}
\nc{\hQ}{{\widehat{Q}}}
\nc{\hsM}{\widehat{\mathsf M}{}}
\nc{\hfM}{\widehat{\mathfrak M}{}}
\nc{\hCP}{\widehat{\mathcal P}{}}
\nc{\hCR}{\widehat{\mathcal R}{}}
\nc{\hCS}{{\widehat{\mathcal S}}}
\nc{\hfZ}{\widehat{\mathfrak Z}{}}
\nc{\hZ}{\widehat{Z}{}}

\nc{\urho}{\underline{\rho}}
\nc{\uB}{\underline{B}}
\nc{\uC}{{\underline{\mathbb{C}}}}
\nc{\ui}{\underline{i}}
\nc{\ofP}{{\overline{\mathfrak{P}}}}

\nc{\hrho}{{\hat{\rho}}}

\nc{\unl}{\underline}
\nc{\ol}{\overline}
\nc{\one}{{\mathbf{1}}}
\nc{\two}{{\mathbf{t}}}

\newcommand{\bbeta}{{\boldsymbol{\beta}}}

\newcommand{\bmu}{{\boldsymbol{\mu}}}

\newcommand{\bomega}{{\boldsymbol{\omega}}}

\nc{\Sym}{{\mathop{\operatorname{Sym}}}}
\nc{\Tot}{{\mathop{\operatorname{\normalshape Tot}}}}
\nc{\Hilb}{{\mathop{\operatorname{\normalshape Hilb}}}}
\nc{\Hom}{{\mathop{\operatorname{Hom}}}}
\nc{\CHom}{{\mathop{\operatorname{{\mathcal{H}}\it om}}}}
\nc{\defi}{{\mathop{\operatorname{\normalshape def}}}}
\nc{\length}{{\mathop{\operatorname{\normalshape length}}}}
\nc{\Rees}{{\mathop{\operatorname{Rees}}}}
\nc{\gr}{{\mathop{\operatorname{gr}}}}

\nc{\Cliff}{{\mathsf{Cliff}}}
\nc{\Fl}{{\mathcal{F}\ell}}
\nc{\Fib}{{\mathsf{Fib}}}
\nc{\Coh}{{\mathsf{Coh}}}
\nc{\FCoh}{{\mathsf{FCoh}}}

\nc{\reg}{{\text{\normalshape reg}}}
\nc{\res}{{\operatorname{res}}}
\nc{\rk}{{\operatorname{rk}}}
\nc{\cplus}{{\mathbf{C}_+}}
\nc{\cminus}{{\mathbf{C}_-}}
\nc{\cthree}{{\mathbf{C}_*}}
\nc{\Qbar}{{\bar{Q}}}

\nc{\bh}{{\bar{h}}}
\nc{\bOmega}{{\overline{\Omega}}}
\nc\tGr{\widetilde{\Gr}}

\nc{\seq}[1]{\stackrel{#1}{\sim}}
\nc\ogu{\overline{G/U}}
\nc\chlam{\check{\lam}}

\nc\St{\operatorname{St}}
\nc{\tF}{\widetilde{\mathcal F}}

\nc{\Gr}{{\on{Gr}}}

\nc{\lmod}{\setminus}
\nc{\Cx}{\BC^\times}

\nc\iso{\,\vphantom{j^{X^2}}\smash{\overset{\sim}{\vphantom{\vstretch{0.3}{A}}\smash{\longrightarrow}}}\,}
\nc\QM{\mathcal{QM}}

\newcommand{\oZ}{\vphantom{j^{X^2}}\smash{\overset{\circ}{\vphantom{\vstretch{0.7}{A}}\smash{Z}}}}
\nc{\chmu}{\check{\mu}}
\newcommand\la{\langle}
\newcommand\ra{\rangle}
\newcommand{\oW}{\overline{\mathcal{W}}{}}

\newcommand{\Cartan}{Y^=}

\begin{document}
\title{Comultiplication for shifted Yangians and quantum open Toda lattice}
\dedicatory{To David Kazhdan on his 70th birthday, with admiration}
\author
[M.~Finkelberg, J.~Kamnitzer, K.~Pham, L.~Rybnikov, and A.~Weekes]
{Michael Finkelberg, Joel Kamnitzer, Khoa Pham, Leonid Rybnikov, and Alex Weekes}

\begin{abstract}
We study a coproduct in type $A$ quantum open Toda lattice in terms of a
coproduct in the shifted Yangian of $\fsl_2$. At the classical level this
corresponds to the multiplication of scattering matrices of euclidean
$SU(2)$ monopoles.  We also study coproducts for shifted Yangians for any simply-laced Lie algebra.
\end{abstract}
\maketitle

\begin{flushright}
{\small{\textit{In my youth I have multiplied too many matrices,\\
so now I try to avoid it if there is a way around.}}\\
(D.~Kazhdan)\hskip 20mm\hphantom{x}}
\end{flushright}

\sec{int}{Introduction}

\ssec{intro1}{The Toda lattice}
Let $G\supset B\supset T$ be a reductive group with a Borel subgroup and
a Cartan subgroup; let $U$ be the
unipotent radical of $B$, and let $\fn$ be the Lie algebra of $U$.
Let $\psi\colon U(\fn)\to\BC$ be a regular character. Let $\CalD(G)$ be
the ring of differential operators on $G$. The action of $\fn$ by the
left-invariant (resp.\ right-invariant) vector fields on $G$ gives rise to
the homomorphism $U(\fn)\otimes U(\fn)\to\CalD(G)$. The ring $\CT(G)$ is
defined as the quantum hamiltonian reduction
$\CalD(G)/\!\!/(U\times U;\psi,-\psi)$. It comes equipped with a homomorphism
from the ring $ZU(\fg)$ of biinvariant differential operators on $G$.
The action of $U\times U$ on the big Bruhat cell
$C_{w_0}=U\cdot T\cdot\dot{w}_0\cdot U$ is free, and the quantum hamiltonian
reduction $\CalD(C_{w_0})/\!\!/(U\times U;\psi,-\psi)$ is isomorphic to the
ring $\CalD(T)$ of differential operators on $T$. Thus we obtain a localization
homomorphism $\CT(G)\hookrightarrow\CalD(T)$, and the composed embedding
$ZU(\fg)\hookrightarrow\CT(G)\hookrightarrow\CalD(T)$. This is the classical
construction of the quantum open Toda lattice due to Kazhdan-Kostant.

At the quasiclassical level, we denote by $\fZ(G)$ the symplectic variety
obtained by the hamiltonian reduction of the cotangent bundle of
$G\colon \fZ(G)=T^*G/\!\!/(U\times U;\psi,-\psi)$. It is equipped with a
lagrangian projection onto $\on{Spec}ZU(\fg)=\fh^*/W$ where $ZU(\fg)$ is
the Harish-Chandra center, $\fh$ is the Lie algebra of $T$, and $W$ is the
Weyl group of $(G,T)$. Furthermore, $\fZ(G)$ contains an open symplectic
subvariety $T^*T$ (the cotangent bundle to the torus $T$), and thus we obtain
the composed lagrangian projection
$\pi_G\colon T^*T\hookrightarrow\fZ(G)\to\fh^*/W$
(Poisson commuting Toda hamiltonians).

\ssec{intro2}{Multiplicative structure}
In case $G=GL(n)$, there is the following explicit construction of the Toda
hamiltonians (see e.g.~\cite[Section~2]{ft} and references therein).
Let $t_1,\ldots,t_n$ be the diagonal matrix elements coordinates on the
diagonal torus $T\subset GL(n)$. Let $w_1,\ldots,w_n$ be the corresponding
coordinates on the dual Lie algebra $\fh^*$. For $r=1,\ldots,n$ we consider
the local Lax matrix $L_r(z)=\left(\begin{array}{cc}z-w_r&t_r\\
-t_r^{-1}&0\end{array}\right)\in SL(2,\BC[z])$, and form the complete
monodromy matrix $L(z)=L_1(z)\cdots L_n(z)=\left(\begin{array}{cc}Q(z)&R'(z)\\
R(z)&Q'(z)\end{array}\right)$. Then the Toda hamiltonians
$\pi_{GL(n)}(t_1,w_1,\ldots,t_n,w_n)$ are nothing but the coefficients of the
polynomial $Q(z)$.\footnote{Realizing $Sp(2n)$ as the folding of $GL(2n)$
and identifying the Siegel Levi subgroup with $GL(n)$, we deduce that the
Toda hamiltonians $\pi_{Sp(2n)}(t_1,w_1,\ldots,t_n,w_n)$ are nothing but the
(even degree) coefficients of $\sQ(z)$ where
$\left(\begin{array}{cc}\sQ(z)&\sR'(z)\\
\sR(z)&\sQ'(z)\end{array}\right)=L_1(z)\cdots L_n(z)L'_n(z)\cdots L'_1(z)$,
and $L'_r(z)=\left(\begin{array}{cc}z+w_r&t_r^{-1}\\
-t_r&0\end{array}\right)\in SL(2,\BC[z]),\ r=1,\ldots,n$.}

Our note stems from a simple observation that the above multiplicative
structure of type $A$ Toda hamiltonians arises from the associative
multiplication $\fZ(GL(k))\times\fZ(GL(l))\to\fZ(GL(k+l))$ that can be
quantized to a coassociative comultiplication
$\CT(GL(k+l))\to\CT(GL(k))\otimes\CT(GL(l))$. More generally, for any pair
of Levi subgroups $T\subset M\subset L\subset G$ we have a homomorphism
$\CT(L)\to\CT(M)$ satisfying the obvious transitivity relations,
see~\refss{homo}.
Turning back to the type $A$ case, note that $\fZ(GL(n))$ is isomorphic to
the open zastava space $\oZ^n$ of degree $n$ based maps from $(\BP^1,\infty)$
to $(\BP^1,\infty)$, aka moduli space of euclidean $SU(2)$-monopoles of
topological charge $n$. Under this isomorphism, the above complete monodromy
matrix goes to the scattering matrix, and the Toda multiplication goes to
the zastava multiplication~\cite[2(vi,xi,xii)]{bfn16} which we learned
of from D.~Gaiotto and T.~Dimofte, see~\reft{classi}.

\ssec{intro3}{Shifted Yangians}
According to~\cite[Appendix~B]{bfn16}, the quantization $\CT(GL(n))$ of $\BC[\fZ(GL(n))]$
is a certain explicit quotient of the shifted Yangian $Y_{-2n}(\fsl_2)$.
One of our main results is that the above comultiplication $\CT(GL(k+l))\to\CT(GL(k))\otimes\CT(GL(l))$
descends from a comultiplication
$Y_{-2k-2l}(\fsl_2)\to Y_{-2k}(\fsl_2)\otimes Y_{-2l}(\fsl_2)$,
see~\reft{quantu}.

Much of this paper is concerned with the study of comultiplication for shifted Yangians (beyond $ \fsl_2$). In \reft{generalcoproduct}, we establish the existence of such a coproduct for any simply-laced Lie algebra $ \fg $.  This generalizes the coproduct on shifted $\mathfrak{gl}_n$--Yangians defined in \cite[Theorem 11.9]{BK}, which is analogous to the dominant type A case of our construction (see \refr{BK coproduct}).  We also study the corresponding multiplicative structure on the classical limit of shifted Yangians, which are the moduli spaces $ \CW_\mu $ which we introduced in \cite{bfn16}, see \refs{classical}.

We must admit that the identification of the quantum Toda for $GL(n)$ and
the shifted Yangian for $\fsl_2$ (purely algebraic objects) goes through a
topological medium: equivariant homology of the affine Grassmannian of $GL(n)$.
According to~\cite[Appendix~A]{bfn16}, the latter convolution ring has a
natural representation in the difference operators on $\fh^*$. As a bonus
we obtain a bispectrality result in~\refp{bispec}.

\ssec{intro4}{Acknowledgements}

This note can be viewed as an appendix to an appendix of~\cite{bfn16}.
Its existence is due to the generous and patient explanations by the authors of
{\em op.~cit.}, D.~Gaiotto, T.~Dimofte, B.~Feigin, R.~Bezrukavnikov,
P.~Etingof, P.~Zinn-Justin, A.~Marshakov, N.~Guay, S.~Gautam and A.~Tsymbaliuk.
We are very grateful to all of them.  We thank H. Nakajima for pointing out a mistake in earlier versions of~Theorem~\ref{thm: coproduct existence}
and~Remark~\ref{rem: coproduct existence}.  We also thank the referee for helpful comments.
The study of L.R.\ has been funded by the Russian Academic Excellence Project `5-100'. The research of M.F.\ was supported by the grant RSF-DFG 16-41-01013.  The research of J.K.\ was funded by NSERC.  This research was supported in part by Perimeter Institute for Theoretical Physics. Research at Perimeter Institute is supported by the Government of Canada through Industry Canada and by the Province of Ontario through the Ministry of Economic Development and Innovation.

\sec{levi}{Quantum Toda for Levi subgroups}

\ssec{recall}{Quantum Toda lattice}
Let $G\supset B\supset T$ be a reductive group with a Borel and Cartan
subgroup. Let $T\subset B_-\subset G$ be the opposite Borel subgroup;
let $U$ (resp.\ $U_-$) be the unipotent radical of $B$ (resp.\ $B_-$).
The Lie algebra of $G$ (resp.\ $U_-$) will be denoted by $\fg$
(resp.\ $\fn_-$). Let $U_\hbar(\fg), U_\hbar(\fn_-)$ stand for the
$\hbar$-universal enveloping algebras of $\fg, \fn_-$.
Let $\psi\colon U_\hbar(\fn_-)\to\BC[\hbar]$ be a homomorphism such that
$\psi(f_\alpha)=1$ for any simple root $\alpha$ (we fix a root generator
$f_\alpha\in\fn_-\subset U_\hbar(\fn_-)$).
Let $\CalD_\hbar(G)$ stand for the global sections of the sheaf of
$\hbar$-differential operators on $G$: it is the smash product of
$U_\hbar(\fg)$ and $\BC[G]$. The action of $\fn_-$ by the left-invariant
(resp.\ right-invariant) vector fields on $G$ gives rise to the homomorphism
$l$ (resp.\ $r)\colon U_\hbar(\fn_-)\to\CalD_\hbar(G)$.
Let $I_\psi\subset\CalD_\hbar(G)$ be the left ideal generated by the
$\hbar$-differential operators of the sort $l(x_1)-\psi(x_1)+r(x_2)+\psi(x_2),\
x_1,x_2\in U_\hbar(\fn_-)$. We consider the quantum hamiltonian reduction
$$\CT_\hbar(G):=(\CalD_\hbar(G)/I_\psi)^{U_-\times U_-}$$
where the first (resp.\ second) copy of $U_-$ acts on $G$ (and hence on
$\CalD_\hbar(G)$) by the left (resp.\ right) translations:
$(u_1,u_2)\cdot g:=u_1gu_2^{-1}$. It is an algebra containing the center
$ZU_\hbar(\fg)$ via the embedding $ZU_\hbar(\fg)\hookrightarrow\CalD_\hbar(G)$
as both left- and right-invariant $\hbar$-differential operators.
This is the classical Kazhdan-Kostant construction of the quantum Toda
lattice, see~\cite{k79}.

\ssec{homo}{Comparison with the Toda lattice for a Levi subgroup}
For each element $w$ of the Weyl group $W=N_G(T)/T$ we choose its lift $\dot{w}$
into the normalizer $N_G(T)$. Let $w_0$ be the longest element of $W$.
The action of $U_-\times U_-$ on the big Bruhat cell
$C_{w_0}:=U_-\cdot T\cdot\dot{w}_0\cdot U_-=
U_-\cdot T\cdot\dot{w}{}_0^{-1}\cdot U_-\subset G$ is free, and hence
the quantum hamiltonian reduction of $\CalD_\hbar(C_{w_0})$ is isomorphic
to $\CalD_\hbar(T)=\CT_\hbar(T)$, a certain localization of $\CT_\hbar(G)$.
Thus we have an embedding
$ZU_\hbar(\fg)\hookrightarrow\CT_\hbar(G)\hookrightarrow\CalD_\hbar(T)$.\footnote{Quite
often the term {\em quantum Toda lattice} refers to the composite embedding
$ZU_\hbar(\fg)\hookrightarrow\CalD_\hbar(T)$.}

More generally,
let $T\subset L\subset P\supset B$ be a Levi subgroup of a parabolic
subgroup of $G$; let $P_-\supset B_-$ be the opposite parabolic subgroup.
We denote by $\fl,\fp,\fp_-$ the Lie algebras of $L,P,P_-$. We denote by
$U^L$ (resp.\ $U^L_-$) the intersection $L\cap U$ (resp.\ $L\cap U_-$),
and we denote by
$U^P$ (resp.\ $U^P_-$) the unipotent radical of $P$ (resp.\ $P_-$).
Finally, we denote by $\fn^L_-,\fn^P_-,\fn^P$ the Lie algebras of
$U^L_-,U^P_-,U^P$. We will also need the subgroups
$U^P_{w_0}:=\dot{w}_0U^P\dot{w}{}_0^{-1},\ U^L_{w_0}:=\dot{w}_0U^L\dot{w}{}_0^{-1}$
and their Lie algebras $\fn^P_{w_0},\ \fn^L_{w_0}$.
The restriction of $\psi$ to $U_\hbar(\fn^P_-)$
(resp.\ $U_\hbar(\fn^P_{w_0}),\ U_\hbar(\fn^L_-),\ U_\hbar(\fn^L_{w_0})$)
will be denoted by
$\psi^P$ (resp.\ $\psi^P_{w_0},\ \psi^L,\ \psi^L_{w_0}$).
Let $\dot{w}{}^L_0$ be a lift of
$w^L_0$ (the longest element in the parabolic Weyl subgroup $W_L\subset W$)
into the normalizer $N_G(T)$. Let $C_{W_Lw_0}\subset G$ be an affine
open subvariety equal
to the union of all the Bruhat cells $C_{ww_0}$ where $w\in W_L$
(it is the preimage of the big Bruhat cell
$(B\cdot\dot{w}_0P\dot{w}{}_0^{-1})/\dot{w}_0P\dot{w}{}_0^{-1}$
in the partial flag variety $G/\dot{w}_0P\dot{w}{}_0^{-1}$).
The action of $U^P_-\times U^P_{w_0}$ on $C_{W_Lw_0}$ is free, and a closed
embedding $\imath_L^G\colon
L\hookrightarrow C_{W_Lw_0},\ g\mapsto g\dot{w}{}^L_0\dot{w}{}_0^{-1}$,
is a cross section of this action giving rise to an isomorphism
\eq{one}
L\times U^P_-\times U^P_{w_0}\iso C_{W_Lw_0},\ (g,x_1,x_2)\mapsto
x_1g\dot{w}{}^L_0\dot{w}{}_0^{-1}x_2^{-1},
\end{equation}
and hence to $L\iso C_{W_Lw_0}/(U^P_-\times U^P_{w_0})$. Hence
$$(\CalD_\hbar(C_{W_Lw_0})/I_{\psi^P,\psi^P_{w_0}})^{U^P_-\times U^P_{w_0}}\iso\CalD_\hbar(L),$$
where $I_{\psi^P,\psi^P_{w_0}}\subset\CalD_\hbar(C_{W_Lw_0})$ is the left ideal
generated by the $\hbar$-differential operators of the sort
$l(x_1)-\psi^P(x_1)+r(x_2)+\psi^P_{w_0}(x_2),\ x_1\in U_\hbar(\fn^P_-),\
x_2\in U_\hbar(\fn^P_{w_0})$. Indeed,~\refe{one} gives rise to an isomorphism
$\CalD_\hbar(C_{W_Lw_0})\iso
\CalD_\hbar(U^P_-)\otimes\CalD_\hbar(L)\otimes\CalD_\hbar(U^P_{w_0})$, but
$(\CalD_\hbar(U^P_-)/I_{\psi^P})^{U^P_-}\simeq\BC$, and
$(\CalD_\hbar(U^P_{w_0})/I_{\psi^P_{w_0}})^{U^P_{w_0}}\simeq\BC$.

Composing the above isomorphism with the restriction to the open subset
$\CalD_\hbar(G)\to\CalD_\hbar(C_{W_Lw_0})$ we obtain a homomorphism
$$(\CalD_\hbar(G)/I_{\psi^P,\psi^P_{w_0}})^{U^P_-\times U^P_{w_0}}\to\CalD_\hbar(L).$$
Furthermore, we obtain a composed homomorphism
$$\tau_G^L\colon
\CT_\hbar(G)=\left((\CalD_\hbar(G)/I_{\psi^P,\psi^P_{w_0}})^{U^P_-\times U^P_{w_0}}/
I_{\psi^L,\psi^L_{w_0}}\right)^{U^L_-\times U^L_{w_0}}\to
(\CalD_\hbar(L)/I_{\psi^L,\psi^L_{w_0}})^{U^L_-\times U^L_{w_0}}=\CT_\hbar(L).$$
For a pair of Levi subgroups
$M\subset L\subset G$ we have $\imath_M^G=\imath_L^G\circ\imath_M^L$, and
hence $\tau_G^M=\tau_G^L\circ\tau_L^M$.

\ssec{taipei}{Type A}
For $G=GL(n)$ we will denote $\CT_\hbar(GL(n))$ by $\CT_\hbar^n$ for short.
We view $GL(n)$ as the group of invertible $n\times n$-matrices.
Let $T$ be the diagonal subgroup, and let $U$ (resp.\ $U_-$) be the
subgroup of lower (resp.\ upper) triangular matrices with 1's on the main
diagonal. We choose a lift $W\to N_G(T),\ w\mapsto\dot{w}$ representing each
$w\in W=\fS_n$ by the corresponding permutation matrix. In particular, for
any Levi $L$, we have $(\dot{w}_0^L)^2=1$.
Let $L$ be the subgroup of block matrices with blocks of sizes
$k,l$ such that $k+l=n$. Then $\tau_{k,l}:=\tau_G^L$ is a homomorphism
from $\CT_\hbar(G)=\CT_\hbar^{k+l}$ to $\CT_\hbar(L)=\CT_\hbar^k\otimes\CT_\hbar^l$.
We obtained a coassociative comultiplication on $\bigoplus_n\CT_\hbar^n$.

Let $V_{\varpi_1}=V=\BC^n$ be the standard (tautological) representation of $GL(n)$ and
$V_{\varpi_1}^*$ be its dual representation. Let $v_1,\ldots,v_n\in V_{\varpi_1}$
be the standard basis (so that $v_n\in V_{\varpi_1}^U,\ v_1\in V_{\varpi_1}^{U_-}$).
Let $v_1^*,\ldots,v_n^*\in V_{\varpi_1}^*$ be the dual basis
(so that $v_1^*\in(V_{\varpi_1}^*)^U,\ v_n^*\in(V_{\varpi_1}^*)^{U_-}$).
Then the functions
$\varDelta'(g):=\langle gv_1,v^*_n\rangle$ and
$\varDelta(g):=\langle g^{-1}v_1,v^*_n\rangle$ on $G$ are
$U_-\times U_-$-invariant hence survive after Hamiltonian reduction. We denote
their images in $\CT_\hbar^n$ by the same symbols $\varDelta'$ and
$\varDelta$ for brevity.

Any central element $C\in U_\hbar(\fg)$ is naturally a bi-invariant $\hbar$-differential operator on $G$ hence survives after Hamiltonian reduction as well. We denote by $C_1$ and $C_2$ the images in $\CT_\hbar^n$ of the linear central element $\sum\limits_{i=1}^ne_{ii}$ and the quadratic central element $\frac{1}{2}\sum\limits_{i,j=1}^n(e_{ii}e_{jj}-e_{ij}e_{ji})$, respectively.

\prop{25} We have $\tau_{k,l}(\varDelta')=1\otimes\varDelta'$, $\tau_{k,l}(\varDelta)=\varDelta\otimes1$, $\tau_{k,l}(C_1)=C_1\otimes1+1\otimes C_1$, $\tau_{k,l}(C_2)=C_2\otimes1+1\otimes C_2+C_1\otimes C_1-\varDelta'\otimes\varDelta-\frac{l\hbar}{2}C_1\otimes1+\frac{k\hbar}{2}1\otimes C_1$.
\eprop

\prf
Straightforward check.
\epr

\ssec{homology}{Equivariant homology of an affine Grassmannian}
Let $G$ be as in~\refss{recall}, let $G^\vee$ be its Langlands dual group,
and let $\Gr_{G^\vee}=G^\vee_\CK/G^\vee_\CO$ be its affine Grassmannian;
here $\CK=\BC((z))\supset\BC[[z]]=\CO$. The affine Grassmannian is acted upon
by a proalgebraic group $G^\vee_\CO\rtimes\BC^\times$
(the second factor acts by loop rotations). The equivariant homology
$H_\bullet^{G^\vee_\CO\rtimes\BC^\times}(\Gr_{G^\vee})$ forms
a convolution algebra, and an isomorphism
$\bbeta\colon H_\bullet^{G^\vee_\CO\rtimes\BC^\times}(\Gr_{G^\vee})\iso\CT_\hbar(G)$
was constructed
in~\cite[Theorem~3]{bf08}.\footnote{An isomorphism $\bbeta$ was constructed in
{\em loc. cit.} for semisimple groups, but the argument works word for word
for reductive groups.} In particular, $\bbeta^{-1}(\hbar)$ is a generator
of $H^\bullet_{\BC^\times}(pt)$, and $\bbeta^{-1}(ZU_\hbar(\fg))=
H^\bullet_{G^\vee_\CO\rtimes\BC^\times}(pt)\subset
H_\bullet^{G^\vee_\CO\rtimes\BC^\times}(\Gr_{G^\vee})$.

For a dominant coweight $\lambda$ of $G^\vee$ the corresponding
$G^\vee_\CO$-orbit closure in $\Gr_{G^\vee}$ is denoted by
$\ol\Gr{}^\lambda_{G^\vee}$, and its intersection cohomology
sheaf is denoted by $\on{IC}^\lambda$. Let $d_\lambda$ stand for
$\dim\ol\Gr{}^\lambda_{G^\vee}$. We have a canonical morphism
from the shifted constant sheaf on $\ol\Gr{}^\lambda_{G^\vee}$ to the intersection
cohomology sheaf:
\eq{cocycle}
\unl\BC{}_{\ol\Gr{}^\lambda_{G^\vee}}[d_\lambda]\to\on{IC}^\lambda,
\end{equation}
and hence $H^{\bullet+d_\lambda}(\ol\Gr{}^\lambda_{G^\vee})\to
H^\bullet(\ol\Gr{}^\lambda_{G^\vee},\on{IC}^\lambda)$ inducing isomorphism in the
top and lowest cohomology.
The geometric Satake isomorphism is an identification
$H^\bullet(\ol\Gr{}^\lambda_{G^\vee},\on{IC}^\lambda)\cong V_\lambda$
with an irreducible representation of $G$ with highest weight
$\lambda$. Note that $V_\lambda$ comes equipped with a highest weight vector
$v_\lambda$ and a lowest weight vector $v_{w_0\lambda}$ arising
from the top and lowest fundamental classes in the cohomology
$H^{\bullet+d_\lambda}(\ol\Gr{}^\lambda_{G^\vee})$. Hence the dual $G$-module
$V_\lambda^*$ comes equipped with a highest weight vector $v^*_{-w_0\lambda}$
and a lowest weight vector $v^*_{-\lambda}$.

We consider the fundamental cycle $[\ol\Gr{}^\lambda_{G^\vee}]\in
H_\bullet^{G^\vee_\CO\rtimes\BC^\times}(\Gr_{G^\vee})$. We want to describe
$\bbeta[\ol\Gr{}^\lambda_{G^\vee}]\in\CT_\hbar(G)$. To this end note that
the matrix coefficient $\langle gv_{w_0\lambda},v^*_{-\lambda}\rangle\in\BC[G]$
is $U_-\times U_-$-invariant and hence gives rise to the same named element
in $\CT_\hbar(G)$.

\lem{beta}
$\bbeta[\ol\Gr{}^\lambda_{G^\vee}]=
\langle gv_{w_0\lambda},v^*_{-\lambda}\rangle\in\CT_\hbar(G)$.
\elem

\prf
Dually to~\refe{cocycle}, we have a canonical morphism from the intersection
cohomology sheaf to the shifted dualizing sheaf
$\on{IC}^\lambda\to\bomega_{\ol\Gr{}^\lambda_{G^\vee}}[-d_\lambda]$ and hence
\begin{multline}
\label{cycle}
c\colon
H^\bullet_{G^\vee_\CO\rtimes\BC^\times}(\ol\Gr{}^\lambda_{G^\vee},\on{IC}^\lambda)\to
H^{\bullet+d_\lambda}_{G^\vee_\CO\rtimes\BC^\times}(\ol\Gr{}^\lambda_{G^\vee},
\bomega_{\ol\Gr{}^\lambda_{G^\vee}})\to\\
\to H^{\bullet+d_\lambda}_{G^\vee_\CO\rtimes\BC^\times}(\Gr_{G^\vee},\bomega_{\Gr_{G^\vee}})=
H_{\bullet+d_\lambda}^{G^\vee_\CO\rtimes\BC^\times}(\Gr_{G^\vee}).
\end{multline}
According to~\cite[Section~3.2,~Lemma~1]{bf08},
$H^\bullet_{G^\vee_\CO\rtimes\BC^\times}(\ol\Gr{}^\lambda_{G^\vee},\on{IC}^\lambda)$
contains a $H^\bullet_{G^\vee_\CO\rtimes\BC^\times}(pt)$-submodule
$H^\bullet_{G^\vee_\CO\rtimes\BC^\times}(pt)\otimes\BC v_{w_0\lambda}$,
and $c(1\otimes v_{w_0\lambda})=[\ol\Gr{}^\lambda_{G^\vee}]$.

Recall the functor $\fF$ of~\cite[6.4]{bf08} from the equivariant derived
category $D_{G^\vee_\CO\rtimes\BC^\times}(\Gr_{G^\vee})$ to the category of asymptotic
Harish-Chandra $U_\hbar(\fg)$-bimodules. It takes $\on{IC}^\lambda$ to the
free bimodule $U_\hbar(\fg)\otimes V_\lambda$, and $\bomega_{\Gr_{G^\vee}}$ to the
hamiltonian reduction $\CK=\kappa_\hbar(U_\hbar(\fg)\otimes\BC[G])$
(see~\cite[Proposition~4]{bf08}). Furthermore, it takes the composed morphism
$\on{IC}^\lambda\to\bomega_{\ol\Gr{}^\lambda_{G^\vee}}[-d_\lambda]\to\bomega_{\Gr_{G^\vee}}[-d_\lambda]$
to the morphism $b\colon U_\hbar(\fg)\otimes V_\lambda\to\CK$ arising from
$U_\hbar(\fg)\otimes V_\lambda\ni
u\otimes v\mapsto u\otimes v\otimes v^*_{-\lambda}\in
U_\hbar(\fg)\otimes V_\lambda\otimes V_\lambda^*\subset
U_\hbar(\fg)\otimes\BC[G]$. Finally, it takes the induced morphism on the
equivariant cohomology to $\kappa_\hbar(b)\colon
\kappa_\hbar(U_\hbar(\fg)\otimes V_\lambda)\to\kappa_\hbar(\CK)$. Hence
$1\otimes v_{w_0\lambda}$ goes to the image of
$1\otimes v_{w_0\lambda}\otimes v^*_{-\lambda}$ in $\kappa_\hbar(\CK)$, that is to
$\langle gv_{w_0\lambda},v^*_{-\lambda}\rangle\in\CT_\hbar(G)$.
\epr

\ssec{clas}{The classical limit}
The quotient algebra $\CT(G):=\CT_\hbar(G)/\hbar$ is a Poisson algebra
containing a maximal Poisson commutative subalgebra $ZU_\hbar(\fg)/\hbar$ ({\em classical
Toda lattice}). We denote the spectrum of the commutative algebra
$\CT(G)$ by $\fZ(G)$. The homomorphism $\tau_G^L$ of~\refss{homo} reduced
modulo $\hbar$ gives rise the the same named homomorphism
$\tau_G^L\colon \CT(G)\to\CT(L)$ and a morphism $\fz_L^G\colon \fZ(L)\to\fZ(G)$.
In the setup of~\refss{taipei} we have
$\CT^n=\CT(GL(n)),\ \fZ^n=\on{Spec}\CT^n$, and we obtain a morphism
$\fz_{k,l}\colon \fZ^k\times\fZ^l\to\fZ^{k+l}$.

The isomorphism $\bbeta$ of~\refss{homology} reduced modulo $\hbar$ gives
rise to the same named isomorphism
$\bbeta\colon H_\bullet^{G_\CO}(\Gr_G)\iso\CT^n$
(cf.\ also~\cite[Theorem~2.12]{bfm05} and~\cite[Theorem~6.3]{t14}).
We also have an isomorphism $\Xi\colon \BC[\oZ^n]\iso H_\bullet^{G_\CO}(\Gr_G)$
of~\cite[Theorem~3.1]{bfn16}
where $\oZ^n$ stands for the open zastava space of degree $n$ based maps
from $(\BP^1,\infty)$ to the flag variety of $SL(2)$ with a marked point of
the upper triangular Borel subgroup. Let $\iota\colon \oZ^n\iso\oZ^n$ be the
Cartan involution of~\cite[1.4(3)]{bdf14}.
We denote the composition
of the above isomorphisms at the level of spectra by
$\Upsilon=\iota\circ\on{Spec}\Xi\circ\on{Spec}\bbeta\colon \fZ^n\iso\oZ^n$.

In elementary terms, $\oZ^n$ is the moduli space of pairs of relatively
prime polynomials $(Q,R)\in\BC[z]$ such that $Q$ is monic of degree $n$,
and $\deg(R)<n$.
The embedding $\Psi\colon \oZ^n\hookrightarrow SL(2,\BC[z])$
of~\cite[2(xi),2(xii)]{bfn16} takes $(Q,R)$ to a unique matrix
$\left(\begin{array}{cc}Q&R'\\ R&Q'\end{array}\right)$ of determinant 1
such that $\deg{R'}<n,\ \deg{Q'}<n-1$. The Cartan involution $\iota$
is intertwined by $\Psi$ with the matrix transposition.
The multiplication morphism
$\bmu\colon \oZ^k\times\oZ^l\to\oZ^{k+l}$ of~\cite[2(vi)]{bfn16} is
intertwined by $\Psi$ with the matrix multiplication:
$\Psi\left(\bmu((Q_1,R_1),(Q_2,R_2))\right)=\Psi(Q_1,R_1)\cdot\Psi(Q_2,R_2)$.
We have the {\em factorization} projection
$\pi\colon \oZ^n\to\BA^{(n)},\ (Q,R)\mapsto Q$, to the configuration space
of unordered roots of $Q$. The subalgebra $\BC[\BA^{(n)}]\subset\BC[\oZ^n]$
corresponds under $\Upsilon$ to the maximal Poisson commutative subalgebra
$ZU(\mathfrak{gl}_n)\subset\CT^n$.
More precisely, let $(e,h,f)$ be an ${\mathfrak{sl}}_2$-triple in
$\fg={\mathfrak{gl}}_n$ such that $e=\left(\begin{array}{ccccc}0&0&0&\ldots&0\\
1&0&0&\ldots&0\\ 0&1&0&\ldots&0\\ \vdots&\vdots&\ddots&\ddots&\vdots\\
0&\ldots&0&1&0 \end{array}\right)$, and $h$ is diagonal. Let $Z_\fg(f)$ be
the centralizer of $f$ in $\fg$, and let $\Sigma^n=e+Z_\fg(f)\subset\fg$ be
a Kostant slice. Let $\fZ_\fg^G$ be the universal centralizer: the moduli space
of pairs of commuting matrices $(x,g)$ such that $x\in\Sigma^n$, and $g$ is
invertible. We have an isomorphism $\zeta\colon \fZ_\fg^G\iso\oZ^n$ taking
a pair $(x,g)$ to $(Q,R)$ where $Q$ is the characteristic polynomial of $x$,
and $R$ is a unique polynomial of degree less than $n$ such that
$R'(x)=g$. In particular,
\eq{R}
R'(z)=g_{n1}z^{n-1}+g_{n-1,1}z^{n-1}+\ldots+g_{21}z+g_{11}.
\end{equation}
We also have an isomorphism $\eta\colon \fZ_\fg^G\iso\fZ^n$ constructed
as follows. The Killing form identifies $e\in\fn$ with $\psi\in\fn_-^*$.
Under this identification, $\fZ_\fg^G\subset G\times\fg\simeq G\times\fg^*=
T^*G$ (left-invariant identification) lies in the moment level
$\mu^{-1}(\psi,-\psi)$ of the moment map $\mu\colon T^*G\to\fn_-^*\times\fn_-^*$.
The composed projection $\fZ_\fg^G\hookrightarrow\mu^{-1}(\psi,-\psi)
\twoheadrightarrow T^*G/\!\!/(U_-\times U_-;\psi,-\psi)=\fZ^n$ is the desired
isomorphism $\eta$. Finally, $\Upsilon=\zeta\circ\eta^{-1}\colon \fZ^n\iso\oZ^n$.

\th{classi}
The following diagram commutes:
$\begin{CD}
\fZ^k\times\fZ^l @>{\fz_{k,l}}>> \fZ^{k+l}\\
@VV{\Upsilon\times\Upsilon}V @V{\Upsilon}VV\\
\oZ^k\times\oZ^l @>{\bmu}>> \oZ^{k+l}
\end{CD}$.
\eth

\prf
The space $\fZ^n$ is equipped with
a Poisson structure by construction (in fact, it is symplectic).
The space $\oZ^n$ is also equipped with a Poisson (symplectic) structure,
see e.g.~\cite{fkmm}. The isomorphism $\Upsilon\colon \fZ^n\iso\oZ^n$ is
Poisson according to~\cite[Proposition~3.18]{bfn16} (more precisely,
both $\on{Spec}\Xi$ and $\iota$ are anti-Poisson, and $\on{Spec}\bbeta$
is Poisson).
The upper arrow in the diagram of~\reft{classi} is a Poisson morphism by
construction.

\lem{mu-poisson}
$\bmu$ is Poisson.
\elem

\prf
The rational $R$-matrix formula for the Poisson bracket $\{T_{ij}(u),T_{kl}(v)\}=\frac{1}{u-v}(T_{il}(u)T_{kj}(v)-T_{il}(v)T_{kj}(u))$ is well-defined on matrix-valued polynomials and compatible with the multiplication of matrices. Note that
$\Psi(\oZ^n)\subset SL(2,\BC[z])$
forms a Poisson subvariety with respect to this bracket. Hence we get a well-defined Poisson structure on $\oZ^n$ compatible with $\bmu$. On the other hand this bracket on $\oZ^n$ is opposite to the standard one~\cite{fkmm}.
Indeed, it follows from~\cite{fkmm}
that $\{Q(u),Q(v)\}=0$, $\{Q(u),R(v)\}=\frac{-1}{u-v}(Q(u)R(v)-Q(v)R(u))$
and $\{R(u),R(v)\}=0$, hence on the coefficients of $Q$ and $R$ the
Poisson brackets in question are opposite. On the other hand the field of
rational functions on $\oZ^n$ is generated by the coefficients of $Q$ and $R$.
\epr

Hence it suffices to check the commutativity of the diagram of~\reft{classi}
on an appropriate
set of Poisson generators of the coordinate rings. Let
$R(z)=\sum\limits_{k=1}^{n}r_kz^{n-k}$, $R'(z)=\sum\limits_{k=1}^{n}r_k'z^{n-k}$ and
$Q(z)=z^n+\sum\limits_{k=1}^{n}q_kz^{n-k}$. Clearly, the functions
$q_k,r_k,r_k'$ (for $k=1,\ldots,n-1$) generate $\BC[\oZ^n]$ as a commutative
ring. We will also need the functions $y_k,y_k'$ for $k=1,2,\ldots$ defined as
$\frac{R(z)}{Q(z)}=\sum\limits_{k=1}^{\infty}y_kz^{-k}$ and
$\frac{R'(z)}{Q(z)}=\sum\limits_{k=1}^{\infty}y'_kz^{-k}$.
Finally, we set $\frac{1}{Q^2(z)}=\sum\limits_{k=1}^\infty x_kz^{-k}$; note that
$x_1=x_2=\ldots=x_{2n-1}=0$, and $x_{2n}=1$.

\lem{lem-cl-gen}
The elements $r_1, r_1',q_1,q_2$ generate the coordinate ring $\BC[\oZ^n]$ as a Poisson algebra.
\elem

\prf
It is easy to see that $y_k,y'_k,x_k,\ k=1,2,\ldots$, generate $\BC[\oZ^n]$
as a commutative ring. We have $y_1=r_1,\ y'_1=r'_1,\ x_{2n+1}=-2q_1,\
x_{2n+2}=3q_1^2-2q_2$.
We have $\{x_{2n+2},y_k\}=2y_{k+1}+2x_{2n+1}y_k$ and
$\{x_{2n+2},y'_k\}=-2y'_{k+1}-2x_{2n+1}y'_k$, and hence all the functions $y_k,y_k'$
are Poisson expressions of the above generators. Also, we have
$\{y_m,y'_l\}=x_{m+l-1}$, and hence by induction we conclude that $x_k$ are
Poisson expressions of the above generators as well.
\epr

Now we can finish the proof of the theorem. By direct computation we have
$\Upsilon^*(r_1)=-\varDelta$, $\Upsilon^*(r_1')=\varDelta'$, $\Upsilon^*(q_1)=C_1$,  $\Upsilon^*(q_2)=C_2$ (we denote the specializations of the elements from $\CT_\hbar^n$ at $\hbar=0$ by the same symbols for brevity). Also, from the $2\times2$-matrix multiplication we see that $\bmu^*(r_1)=r_1\otimes1$, $\bmu^*(r_1')=1\otimes r_1'$, $\bmu^*(q_1)=q_1\otimes1+1\otimes q_1$, $\bmu^*(q_2)=q_2\otimes1+1\otimes q_2+q_1\otimes q_1+r_1'\otimes r_1$. So the theorem follows from \refp{25}.
\epr

\rem{SL}
We assume now $G$ is semisimple, and $T\subset L\subset G$ is a Levi subgroup
containing a Cartan torus. The classical Toda system is the
projection $\fZ(L)\to\fh^*/W_L$ (where $\fh$ is the Lie algebra of $T$).
Let $\Gamma_L^G\subset\fZ(L)\times\fZ(G)$ be the graph of $\fz_L^G$.
Then the projection $\on{pr}_{L,G}\colon \Gamma_L^G\to\fh^*/W_L\times\fh^*/W_G$
is generically finite. If $L\subset G=SL(n)$ corresponds to a decomposition
$n=k+l$, the degree of $\on{pr}_{L,G}$ equals $\binom{n-2}{k-1}$ (D.~Gaiotto,
private communication). If $L=T\subset G=SL(n)$, the degree of
$\on{pr}_{T,SL(n)}$ equals $1,1,2,4,11,33,120,470,2107,10189$ for
$n=2,3,4,5,6,7,8,9,10,11$ respectively (E.~Rains, private communication).
\erem

\ssec{some}{Some equivariant homology classes}
We will need a special case of~\refl{beta}.
Recall the setup of~\refss{homology}, and let $G=GL(n)\cong G^\vee$.
Let $\Gr_G^{\varpi_1}$ (resp.\ $\Gr_G^{-\varpi_1}$) be the closed $G_\CO$-orbit in
$\Gr_G$ formed by all the lattices $L$ such that $z\CO^n\subset L\subset\CO^n$
and $\dim_\BC\CO^n/L=1$ (resp.\ $\CO^n\subset L\subset z^{-1}\CO^n$ and
$\dim_\BC L/\CO^n=1$). The fundamental class of $\Gr_G^{\varpi_1}$
(resp.\ $\Gr_G^{-\varpi_1}$) in $H_\bullet^{G_\CO\rtimes\BC^\times}(\Gr_G)$ will be
denoted $[\Gr_G^{\varpi_1}]$ (resp.\ $[\Gr_G^{-\varpi_1}]$).

Recall the elements $\varDelta',\varDelta\in\CT_\hbar^n$ introduced
in~\refss{taipei}.

\lem{betas} We have $\bbeta[\Gr_G^{\varpi_1}]=\varDelta'$ and $\bbeta[\Gr_G^{-\varpi_1}]=(-1)^{n-1}\varDelta$.
\elem

\prf
This follows from \refl{beta} up to a multiplicative constant.
To determine the constants,
recall that we have chosen the principal nilpotent $e\in\fn$ to be the sum of elementary matrices $e=\sum\limits_{i=1}^n e_{i+1,i}\in\fg=\fgl_n$. The operator of multiplication by the first Chern class of the tautological $G^\vee$-bundle is identified with the action of the principal nilpotent $e\in\fg$ via the isomorphism $H^\bullet_{G^\vee_\CO\rtimes\BC^\times}(\ol\Gr{}^\lambda_{G^\vee},\on{IC}^\lambda)=V_\lambda$. For $\lambda=\pm\varpi_1$ we have $\ol\Gr{}^\lambda_{G^\vee}\simeq\BP^{n-1}$,
hence the top-dimensional fundamental class of $\ol\Gr{}^\lambda_{G^\vee}$ gets identified with $e^{n-1}v_{w_0\lambda}$. The latter is $v_n$ for $\lambda=\varpi_1$,
and $(-1)^{n-1}v_1^*$ for $\lambda=-w_0\varpi_1$.
\epr

\ssec{bis}{Bispectrality}
This section is not used in what follows.

We will need the homomorphism $(\iota_*)^{-1}$ from
$H_\bullet^{G_\CO\rtimes\BC^\times}(\Gr_G)$ to a certain localization of
$\left(H_\bullet^{T_\CO\rtimes\BC^\times}(\Gr_T)\right)^{\fS_n}$, see~\cite[A(i)]{bfn16}.
The convolution algebra $\CA_\hbar(T,0):=H_\bullet^{T_\CO\rtimes\BC^\times}(\Gr_T)$ is a
$\BC[\hbar]$-algebra generated by $w_r,\sfu_r^{\pm1},\ 1\leq r\leq n$,
with relations $[\sfu_r^{\pm1},w_s]=\pm\delta_{r,s}\hbar\sfu^{\pm}_r$.
It satisfies the Ore condition with respect to the set
$\{w_r-w_s+m\hbar,\ 1\leq r\ne s\leq n,\ m\in\BZ\}$, and the corresponding localization is
denoted $\tilde\CA_\hbar$. So we have the homomorphism
$(\iota_*)^{-1}\colon H_\bullet^{G_\CO\rtimes\BC^\times}(\Gr_G)\to
(\tilde\CA_\hbar)^{\fS_n}$.

We define a certain space of formal functions on the diagonal torus
$T\subset G$, containing Whittaker functions.
Let $t_1,\ldots,t_n$ be the diagonal matrix elements considered
as functions on the diagonal torus $T$. Let $\fh$ (resp.\ $\fb$) stand for
the Lie algebra of $T$ (resp.\ $B$). Let $\fR:=
\BC(\fh^*\times\BA^1_{t_1}\times\BA^1_\hbar)[[t_1t_2^{-1},\ldots,t_{n-1}t_n^{-1}]]$
be the ring of formal Taylor series in $t_1t_2^{-1},\ldots,t_{n-1}t_n^{-1}$
with coefficients in the field of rational functions on the product of $\fh^*$
and a line with coordinate $t_1$ and a line with coordinate $\hbar$.
The ring $\CT_\hbar^n\subset\CalD_\hbar(T)$ acts on $\fR$ naturally.
Also, the ring $\tilde\CA_\hbar$ acts on $\fR$ via
its action on $\BC(\fh^*\times\BA^1_\hbar)$ by the difference operators.
Namely, let $w_1=e_{11},\ldots,w_n=e_{nn}$ be the elementary matrices considered
as elements of the diagonal Cartan Lie algebra $\fh$, i.e.\ linear functions
on $\fh^*$. Then we set
$\sfu_r^{\pm1}(w_s)=w_s\pm\hbar\delta_{rs},\ 1\leq r,s\leq n$.
Let $\CR$ be a free rank one $\fR$-module with generator
$t_1^{w_1/\hbar}\cdots t_n^{w_n/\hbar}$. Then $\CT_\hbar^n\subset\CalD_\hbar(T)$
and $\tilde\CA_\hbar$ act on $\CR$.

We consider the generic universal Verma module
$\CM_\hbar(-\rho)=U_\hbar(\fg)\otimes_{U_\hbar(\fb)}\BC(\fh^*\times\BA^1_\hbar)(-\rho)$
where $\BC(\fh^*\times\BA^1_\hbar)(-\rho)$ is a $U_\hbar(\fb)$-module which
factors through the $U_\hbar(\fh)=\BC[\fh^*\times\BA^1_\hbar]$-module where
$x\in\fh$ acts by multiplication by $x-\hbar\rho(x)$ (and $\rho\in\fh^*$ is
the halfsum of the positive roots). It is equipped with the
$\BC(\fh^*\times\BA^1_\hbar)$-valued Shapovalov form $(,)$. The
$\BC(\fh^*\times\BA^1_\hbar)$-vector space of $(U_\hbar(\fn_-),\psi)$-coinvariants
in $\CM_\hbar(-\rho)$ is 1-dimensional, and any coinvariant is proportional
to the Shapovalov scalar product with the Whittaker vector $\fw$ in a
completion of $\CM_\hbar(-\rho)$. More precisely, let
$\psi_+\colon U_\hbar(\fn)\to\BC[\hbar]$ be a homomorphism such that
$\psi_+(e_{i+1,i})=1$ for any $i=1,\ldots,n-1$. Then there is a unique vector
$\fw=\sum_{\unl{d}\in\BN^{n-1}}\fw_{\unl{d}}\in\widehat\CM_\hbar(-\rho)$
(an infinite sum of the weight components)
such that the highest weight component $\fw_{\unl0}$ of $\fw$ is
$1\in\BC(\fh^*\times\BA^1_\hbar)$, and $u\fw=\psi_+(u)\fw$ for any
$u\in U_\hbar(\fn)$. Finally, the Whittaker function $\fW$ is defined as the
Shapovalov scalar product $\fW:=\prod_{r=1}^nt_r^{w_r/\hbar+r-1}
\sum_{\unl{d}\in\BN^{n-1}}(\fw_{\unl{d}},\fw_{\unl{d}})
\prod_{s=1}^{n-1}(t_s/t_{s+1})^{d_s}\in\CR$.

\prop{bispec}
We have $\bbeta(h)\fW=(\iota_*)^{-1}(h)\fW$ for any
$h\in H_\bullet^{G_\CO\rtimes\BC^\times}(\Gr_G)$.
\eprop

\prf
We have $C_1\fW=(\sum_{r=1}^nw_r)\fW,\
C_2\fW=(\sum_{1\leq r<s\leq n}w_rw_s+(\rho_n,\rho_n)\hbar^2)\fW$
(notations of~\refss{taipei};
$\rho_n:=(\frac{n-1}{2},\frac{n-3}{2},\ldots,\frac{1-n}{2})$).
Also, we have
$\varDelta'\fW=(\sum_{r=1}^n\prod_{s\ne r}(w_r-w_s)^{-1}\sfu_r)\fW,\
\varDelta\fW=(\sum_{r=1}^n\prod_{s\ne r}(w_r-w_s)^{-1}\sfu_r^{-1})\fW$
according to~\cite[Theorem~3 and~(6.7)]{vde15}.
According to~\cite[(A.3),(A.4)]{bfn16},
$(\iota_*)^{-1}[\Gr_G^{\varpi_1}]=\sum_{r=1}^n\prod_{s\ne r}(w_r-w_s)^{-1}\sfu_r,\
(\iota_*)^{-1}[\Gr_G^{-\varpi_1}]=\sum_{r=1}^n\prod_{s\ne r}(w_s-w_r)^{-1}\sfu_r^{-1}$.
Since $C_1,C_2,\varDelta',\varDelta$ generate $\CT^n_\fh[\hbar^{-1}]$,
the proposition follows from~\refl{betas}.
\epr

\sec{shifted-yangian}{Shifted Yangians}
In this section we consider the family of algebras known as shifted Yangians, following \cite[Appendix B]{bfn16}.  Our main goal is to prove a PBW theorem for these algebras.

Let $ \fg $ denote a simply-laced semisimple Lie algebra (of finite type) with simple roots $ \{ \alpha_i \}_{i \in I} $.  We write $ \alpha_i \cdot \alpha_j $ for the usual inner product of these simple roots.

\begin{Def}
The {\em Cartan doubled Yangian} $Y_\infty := Y_\infty(\fg)$ is defined to be the $ \BC $--algebra with generators $ E_i^{(q)}, F_i^{(q)}, H_i^{(p)} $ for $ i\in I$, $ q > 0 $ and $ p \in \BZ $, with relations
\begin{align*}
[H_i^{(p)}, H_j^{(q)}] &= 0,  \\
[E_i^{(p)}, F_j^{(q)}] &=  \delta_{ij} H_i^{(p+q-1)}, \\
[H_i^{(p+1)},E_j^{(q)}] - [H_i^{(p)}, E_j^{(q+1)}] &= \frac{\alpha_i \cdot \alpha_j}{2} (H_i^{(p)} E_j^{(q)} + E_j^{(q)} H_i^{(p)}) , \\
[H_i^{(p+1)},F_j^{(q)}] - [H_i^{(p)}, F_j^{(q+1)}] &= -\frac{\alpha_i \cdot \alpha_j}{2} (H_i^{(p)} F_j^{(q)} + F_j^{(q)} H_i^{(p)}) , \\
[E_i^{(p+1)}, E_j^{(q)}] - [E_i^{(p)}, E_j^{(q+1)}] &= \frac{\alpha_i \cdot \alpha_j}{2} (E_i^{(p)} E_j^{(q)} + E_j^{(q)} E_i^{(p)}), \\
[F_i^{(p+1)}, F_j^{(q)}] - [F_i^{(p)}, F_j^{(q+1)}] &= -\frac{\alpha_i \cdot \alpha_j}{2} (F_i^{(p)} F_j^{(q)} + F_j^{(p)} F_i^{(q)}),\\
i \neq j, N = 1 - \alpha_i \cdot \alpha_j \Rightarrow
\operatorname{sym} &[E_i^{(p_1)}, [E_i^{(p_2)}, \cdots [E_i^{(p_N)}, E_j^{(q)}]\cdots]] = 0, \\
i \neq j, N = 1 - \alpha_i \cdot \alpha_j \Rightarrow
\operatorname{sym} &[F_i^{(p_1)}, [F_i^{(p_2)}, \cdots [F_i^{(p_N)}, F_j^{(q)}]\cdots]] = 0.
\end{align*}
\end{Def}

\rem{Kac-Moody case}
Although we have assumed $\fg$ to be simply-laced and finite type, we expect that our results hold in greater generality (e.g. certain Kac-Moody algebras).  Indeed, we make use of two results for the ordinary Yangian $Y = Y(\fg)$: (1) the PBW theorem for $Y$, and (2) the existence of a coproduct $\Delta: Y\rightarrow Y\otimes Y$. If $\fg$ is such that (1) and (2) are known to hold, then the results in this section and the next should hold.  We will not pursue this direction further here.
\erem

We denote by $Y_\infty^{>}, Y_\infty^{\geq}$ the subalgebras of $Y_\infty$ generated by the $E_i^{(q)}$ (resp. $E_i^{(q)}$ and $H_i^{(p)}$).  Likewise we denote by $Y_\infty^<, Y_\infty^{\leq}$ the subalgebras generated by the $F_i^{(q)}$ (resp. $F_i^{(q)}, H_i^{(p)}$).  Also denote by $\Cartan_\infty$ the subalgebra generated by the $H_i^{(p)}$.  We will use similar notation for the various quotients of $Y_\infty$ that we define below.

\rem{maps from Yangian}
There is a surjective homomorphism $Y^> \twoheadrightarrow Y_\infty^>$, where $Y^>$ is the analogous subalgebra of the ordinary Yangian $Y$.  This map is defined by $E_i^{(q)}\mapsto E_i^{(q)}$.
\erem

\rem{egamma}
Consider a positive root $\beta$, and pick any decomposition $\beta = \alpha_{i_1} +\ldots + \alpha_{i_\ell}$ into simple roots so that the element $ [e_{i_1}, [e_{i_2}, \ldots, [ e_{i_{\ell-1}}, e_{i_\ell}]\cdots] $ is a non-zero element of the root space $\fg_\beta$.  Consider also $q>0$ and a decomposition $q + \ell - 1 = q_1 +\ldots + q_\ell$ into positive integers.  Then we define a corresponding element of $Y_\infty$:
\eq{eq: other root vectors}
E_\beta^{(q)} :=  [ E_{i_1}^{(q_1)}, [E_{i_2}^{(q_2)},\ldots [E_{i_{\ell-1}}^{(q_{\ell-1})}, E_{i_\ell}^{(q_\ell)}]\cdots ].
\end{equation}
This element, which we call a PBW variable, depends on the choices above.  However, we will fix arbitrarily such a choice for each $ \beta $ and $ q $.

Similarly, we define PBW variables $F_\beta^{(q)}$.
\erem

\defe{def: shifted yangians}
For any coweight $\mu$, the shifted Yangian $Y_\mu$ is defined to be the quotient of $Y_\infty$ by the relations $H_i^{(p)}=0$ for all $p<-\langle \mu,\alpha_i\rangle$ and $H_i^{(-\langle \mu,\alpha_i\rangle)}=1$.
\edefe

\rem{Ordinary Yangian}
The algebra $ Y = Y_0 $ is the usual Yangian with its standard Drinfeld presentation (except that upper indices are shifted by 1).
\erem

\rem{maps from Yangian 2}
The homomorphism from \refr{maps from Yangian} can be extended to a surjection $Y^\geq \twoheadrightarrow Y_\mu^\geq$, defined by $E_i^{(q)}\mapsto E_i^{(q)}$ and $H_i^{(q)} \mapsto H_i^{(-\langle \mu,\alpha_i\rangle+q)}$ for $q>0$.  Similarly, there is a surjection $Y^\leq \twoheadrightarrow Y_\mu^\leq$.
\erem

There are natural ``shift homomorphisms'' between these algebras:
\prop{shiftmaps}
Let $\mu$ be a coweight, and $\mu_1,\mu_2$ be antidominant coweights. Then there exists a homomorphism $\iota_{\mu,\mu_1,\mu_2}: Y_\mu \longrightarrow Y_{\mu+\mu_1+\mu_2}$ defined by
\eq{shiftmaps}
H_i^{(r)}\mapsto H_i^{(r- \langle \mu_1+\mu_2,\alpha_i\rangle)},\ \ E_i^{(r)}\mapsto E_i^{(r-\langle \mu_1,\alpha_i\rangle)},\ \ F_i^{(r)}\mapsto F_i^{(r-\langle \mu_2 ,\alpha_i\rangle)}.
\end{equation}
\eprop

\prf Immediate from~\refd{def: shifted yangians}.
\epr

\rem{remark:shiftmaps}
Given our present conventions, for $\mu$ dominant, the shifted Yangian as defined in \cite[Section 3.6]{ktwwy} as a subalgebra of $ Y$, can be identified with the image of the shift homomorphism $\iota_{\mu,0, -\mu}\colon Y_\mu \longrightarrow Y_0 = Y$ (here we use \refc{cor: injectivity of shift maps} to see that the map is injective).  On the other hand, when $ \mu $ is not dominant, these shifted Yangians are not subalgebras of $ Y $ and their definition first appeared \cite[Appendix B]{bfn16}.

The definition of shifted Yangians was originally inspired by the work of Brundan-Kleshchev \cite{BK}, who considered shifted Yangians inside the $\mathfrak{gl}_n$ Yangian.
\erem

\begin{Rem}\label{lnH}
Set $S_i^{(-\langle \mu,\alpha_i\rangle +1)} =H_i^{(-\langle \mu,\alpha_i\rangle +1)}$ and
\eq{def: S_i}
S_i^{(-\langle \mu,\alpha_i\rangle +2)} = H_i^{(-\langle \mu,\alpha_i\rangle +2)} - \tfrac{1}{2}\big(H_i^{(-\langle \mu,\alpha_i\rangle +1)}\big)^2
\end{equation}
For $r\geq 1$, it is not hard to check that
\begin{align*}
[S_i^{(-\langle \mu, \alpha_i \rangle +2)} ,E_j^{(r)}]&=(\alpha_i \cdot \alpha_j)  E_j^{(r+1)}, \\
[S_i^{(-\langle \mu, \alpha_i \rangle +2)} ,F_j^{(r)}]&=-(\alpha_i \cdot \alpha_j)  F_j^{(r+1)}.\\
\end{align*}
\end{Rem}

\begin{Lem}\label{4generators}
Let $\mu$ be an antidominant coweight. As a unital associative algebra, $Y_\mu$ is generated by $E_i^{(1)},F_i^{(1)},S_i^{(-\langle \mu,\alpha_i\rangle + 1)}=H_i^{(-\langle \mu,\alpha_i\rangle+1)}$ and $S_i^{(-\langle \mu,\alpha_i\rangle +2)}=H_i^{(-\langle \mu,\alpha_i\rangle+2)} -\frac{1}{2}(H_i^{(-\langle \mu,\alpha_i\rangle+1)})^2$. Alternatively, $Y_\mu$ is also generated by $E_i^{(1)},F_i^{(1)}, H_i^{(-\langle \mu,\alpha_i\rangle +k)}$ ($k=1,2$). In particular, $Y_\mu$ is finitely generated.
\end{Lem}
\begin{proof}
For the first assertion, it is enough to show that $E_i^{(r)},F_i^{(r)}H_i^{(s)}$ lie in the subalgebra generated by $E_i^{(1)},F_i^{(1)}, S_i^{(-\langle \mu,\alpha_i\rangle +k)}$ ($k=1,2$)  for all $r\geq1, s\geq -\langle \mu,\alpha_i\rangle +1$.
This is clear since $E_i^{(r)}=\frac{1}{2}[S_i^{(-\langle \mu,\alpha_i\rangle +2)},E_i^{(r-1)}]$, $F_i^{(r)}=-\frac{1}{2}[S_i^{(-\langle \mu,\alpha_i\rangle +2)},F_i^{(r-1)}]$ for all $r \geq 2$ and since $H_i^{(s)}= [E_i^{(1)},F_i^{(s)}]$ for all $s\geq -\langle \mu,\alpha_i\rangle +1$.

The second assertion follows immediately from the first since the subalgebra generated by $E_i^{(1)}, F_i^{(1)},S_i^{(-\langle \mu,\alpha_i\rangle +k)}$ ($k=1,2$) is contained in the subalgebra generated by the $E_i^{(1)}, F_i^{(1)},H_i^{(-\langle \mu,\alpha_i\rangle +k)}$ ($k=1,2$).
\end{proof}

\ssec{PBW Theorem}{PBW Theorem}

In this section we will prove the PBW theorem for the algebras $Y_\mu$.  This generalizes the well-known case of the ordinary Yangian $Y $ due to Levendorskii \cite{L2}, as well as the case when $\mu$ is dominant \cite[Proposition 3.11]{kwwy}.

For each positive root $\beta$ and $q>0$, consider elements $E_\beta^{(q)}, F_\beta^{(q)} \in Y_\mu$ defined as images under $Y_\infty\twoheadrightarrow Y_\mu$ of those described in \refr{egamma}.  Choose a total order on the set of PBW variables
\eq{total-order-generators}
\left\{ E_\beta^{(q)} : \beta\in \Delta^+, q>0\right\} \cup \left\{ F_\beta^{(q)} : \beta\in \Delta^+, q>0 \right\} \cup \left\{ H_i^{(p)} : i\in I, p> -\la\mu, \alpha_i\ra \right\}
\end{equation}
If $ \mu = 0 $, then ordered monomials in these PBW variables form a basis of $ Y $ by \cite{L2}.

For simplicity we will assume that we have chosen a block order with respect to the three subsets above, i.e. ordered monomials have the form $EFH$.
\prop{PBW monomials span}
$Y_\mu$ is spanned by ordered monomials in the PBW variables.
\eprop
\prf
We first claim that $Y_\mu$ is spanned by elements of the form $ x y$, with $x\in Y_\mu^>$ and $y\in Y_\mu^\leq$.  Indeed, the relation $[E_i^{(p)}, F_j^{(q)}] = \delta_{ij} H_i^{(r+s-1)}$ allows us to reorder products of the generators $E_i^{(p)}$ and $F_j^{(q)}$, while the surjection $Y^\geq \rightarrow Y_\mu^\geq$ from \refr{maps from Yangian 2} shows that we may reorder products of the generators $E_i^{(p)}$ and $ H_i^{(q)}$ (as this is true for the ordinary Yangian $Y$).

Next, we claim that $Y_\mu^>$ is spanned by ordered monomials in the elements $E_\beta^{(q)}$.  To see this, we can again appeal to the surjection $Y^>\rightarrow Y_\mu^>$ together with the PBW theorem for $Y^>$.  Similarly, $Y_\mu^\leq$ is spanned by ordered monomials in the elements $F_\beta^{(q)}$ and $H_i^{(p)}$. Altogether, this proves the proposition.
\epr

\th{th: PBW antidominant}
Let $\mu$ be antidominant.  Then the set of ordered monomials in the PBW variables forms a basis for $Y_\mu$ over $\BC$.
\eth
The proof of the above theorem will be given in \refss{The proof of PBW}.  For now we give two important corollaries:

\cor{cor: PBW general}
For $\mu$ arbitrary, the set of ordered monomials in the PBW variables forms a basis for $Y_\mu$ over $\BC$.
\ecor
\prf
We may choose a shift homomorphism $\iota\colon Y_\mu \longrightarrow Y_{\mu'}$ from  \refp{shiftmaps} such that $\mu'$ is antidominant.  Under $\iota$, the images of the elements $E_\beta^{(q)}, F_\beta^{(q)} \in Y_\mu$ have the form \refe{eq: other root vectors}, i.e. we may consider these images as PBW variables for $Y_{\mu'}$.  In particular, we may extend these images to a full set of PBW variables  for $Y_{\mu'}$.

It follows from \reft{th: PBW antidominant} that the set of ordered monomials in the PBW variables for $Y_\mu$ map bijectively under $\iota$ onto a subset of a basis for $Y_{\mu'}$, so in particular they are linearly independent in $Y_\mu$.  Combined with \refp{PBW monomials span}, this implies that the PBW monomials for $Y_\mu$ form a basis (and also that $\iota$ is injective).
\epr

\cor{cor: injectivity of shift maps}
For any $\mu$ and antidominant $\mu_1,\mu_2$, the shift homomorphism $\iota_{\mu,\mu_1,\mu_2}\colon Y_\mu \longrightarrow Y_{\mu+\mu_1+\mu_2}$ from \refp{shiftmaps} is injective.
\ecor
\prf
Analogous to the proof of the previous corollary.
\epr

\ssec{An auxilliary algebra}{The algebra $\widetilde{Y}$}

\defe{def: Y tilde}
The algebra $\widetilde{Y}$ is defined to be the quotient of $Y_\infty$ by the relations $H_i^{(p)} = 0$ for all $i\in I$ and $p<0$.
\edefe

To distinguish the generators of $\widetilde{Y}$, we will denote them by $\widetilde{E}_i^{(q)}, \widetilde{F}_i^{(q)}, \widetilde{H}_i^{(p)}$.  The following result is a key tool in proving \reft{th: PBW antidominant}:
\lem{PBW Y tilde} \mbox{}
\begin{itemize}
\item[(a)] There is an embedding of algebras $\widetilde{Y}\hookrightarrow Y\otimes_\BC \BC[H_i^{(0)}: i\in I] $, defined by
\eq{eq: embedding Y tilde}
\widetilde{E}_i^{(q)} \longmapsto E_i^{(q)} \otimes H_i^{(0)}, \ \widetilde{F}_i^{(q)} \longmapsto F_i^{(q)}\otimes 1, \ \widetilde{H}_i^{(p)} \longmapsto H_i^{(p)} \otimes H_i^{(0)}.
\end{equation}
\item[(b)] Ordered monomials in the elements of the set
\eq{total-order-generators-tilde}
\left\{ \widetilde{E}_\beta^{(q)} : \beta\in\Delta^+, q>0\right\} \cup \left\{ \widetilde{F}_\beta^{(q)} : \beta\in\Delta^+, q>0\right\} \cup \left\{ \widetilde{H}_i^{(p)} : i\in I, p\geq 0\right\}
\end{equation}
form a basis for $\widetilde{Y}$ over $\BC$.
\end{itemize}
\elem
\prf
Using the relations for $\widetilde{Y}$ as inherited from $Y_\infty$, one can verify that \refe{eq: embedding Y tilde} is a homomorphism.

 The remainder of the proof is analogous to that of \refc{cor: PBW general}: first one shows that PBW monomials span $\widetilde{Y}$ by using the relations among its generators, and second one observes that these monomials map bijectively under \refe{eq: embedding Y tilde} onto a subset of a basis for $Y\otimes_\BC \BC[H_i^{(0)}]$ (using the PBW theorem for $Y$).
\epr

\cor{freeness-over-polynomials}
Let $\mu$ be an antidominant coweight.  Then $\widetilde{Y}$ is free as a right module over the polynomial ring
\begin{equation}
\BC[ \widetilde{H}_i^{(s)} : i\in I, 0\leq s\leq - \la \mu,\alpha_i\ra ],
\end{equation}
with basis consisting of ordered monomials in the set
\eq{total-order-generators-tilde-mu}
\{ \widetilde{E}_\gamma^{(r)} : \gamma\in \Delta_+, r\geq 1\} \cup \{ \widetilde{F}_\gamma^{(r)}: \gamma\in \Delta_+, r\geq 1\} \cup \{ \widetilde{H}_i^{(s)} : i\in I, s > -\la\mu,\alpha_i \ra \}.
\end{equation}
\ecor

\ssec{The proof of PBW}{The PBW Theorem in the antidominant case}

Let $\mu$ be antidominant.  The following result is immediate from the definitions of $Y_\mu$ and $\widetilde{Y}$ as quotients of $Y_\infty$:
\lem{Y tilde to Y mu}
There is a surjective homomorphism $\widetilde{Y} \twoheadrightarrow Y_\mu$ defined by $\widetilde{E}_i^{(q)}\mapsto E_i^{(q)}, \widetilde{F}_i^{(q)}\mapsto F_i^{(q)}$ and $\widetilde{H}_i^{(p)}\mapsto H_i^{(p)}$.  The kernel of this homomorphism is the ideal
\begin{equation}
I_\mu := \left\la \widetilde{H}_i^{(s)} -\delta_{s,-\la\mu,\alpha_i\ra} : i\in I, 0\leq s \leq -\la \mu, \alpha_i\ra \right\ra_{\text{two-sided}}.
\end{equation}
\elem

\lem{two-sided=left}
$I_\mu$ is equal to the left ideal
\eq{left-ideal}
\left\la\widetilde{H}_i^{(s)} -\delta_{s,-\la\mu,\alpha_i\ra} : i\in I, 0\leq s \leq -\la \mu, \alpha_i\ra \right\ra_{\mathrm{left}}.
\end{equation}
\elem
\prf
Denote this left ideal by $I_\mu^{\text{left}}$.  We will prove the claim by showing that $I_\mu^{\text{left}}$ is also a right ideal.

In $\widetilde{Y}$, we have the relations
\eq{rels-1}
[\widetilde{H}_i^{(0)}, \widetilde{E}_j^{(r)}] = 0, \ \ [\widetilde{H}_i^{(1)}, \widetilde{E}_j^{(r)}] = (\alpha_i, \alpha_j)  \widetilde{E}_j^{(r)} \widetilde{H}_i^{(0)},
\end{equation}
\eq{rels-2}
[\widetilde{H}_i^{(s+1)}, \widetilde{E}_j^{(r)}] = [\widetilde{H}_i^{(s)}, \widetilde{E}_j^{(r+1)}] + \tfrac{(\alpha_i,\alpha_j)}{2} [\widetilde{H}_i^{(s)}, \widetilde{E}_i^{(r)}] + (\alpha_i,\alpha_j)  \widetilde{E}_j^{(r)} \widetilde{H}_i^{(s)}.
\end{equation}
By induction on $s$, it follows that for any $r\geq 1$,
\eq{rels-3}
[\widetilde{H}_i^{(s)}, \widetilde{E}_j^{(r)}] \in \left\la \widetilde{H}_i^{(0)},\ldots,\widetilde{H}_i^{(s-1)} \right\ra_{\text{left}}.
\end{equation}
For $0\leq s\leq -\la\mu,\alpha_i\ra$, we therefore have
\begin{equation}
\big(\widetilde{H}_i^{(s)} -\delta_{s,-\la\mu,\alpha_i\ra} \big) \widetilde{E}_j^{(r)} \in \widetilde{E}_j^{(r)}  \big(\widetilde{H}_i^{(s)} -\delta_{s,-\la\mu,\alpha_i\ra} \big) + I_\mu^{\text{left}} = I_\mu^{\text{left}}.
\end{equation}
It follows that right multiplication by $\widetilde{E}_j^{(r)}$ preserves $I_\mu^{\text{left}}$.  Similarly for $\widetilde{F}_j^{(r)}$, while for $\widetilde{H}_j^{(r)}$ this is clear.   These elements generate $\widetilde{Y}$, so $I_\mu^{\text{left}}$ is a right ideal.
\epr

\begin{proof}[Proof of \reft{th: PBW antidominant}]
Consider the homomorphism
\eq{map-polynomial-ring}
\BC[\widetilde{H}_i^{(s)}: i\in I, 0\leq s\leq -\la\mu,\alpha_i\ra] \longrightarrow \BC,
\end{equation}
defined by $\widetilde{H}_i^{(s)}\mapsto \delta_{s,-\la\mu,\alpha_i\ra}$.  By \refl{two-sided=left}, we see that $Y_\mu$ is the base change of the right module $\widetilde{Y}$ with respect to the map \refe{map-polynomial-ring}:
\begin{equation}
Y_\mu = \widetilde{Y} \otimes_{\BC[\widetilde{H}_i^{(s)}: i\in I, 0 \leq s \leq -\la\mu,\alpha_i\ra]} \BC.
\end{equation}
Recall that $\widetilde{Y}$ is a free right module over $\BC[\widetilde{H}_i^{(s)}: i\in I, 0 \leq s \leq -\la\mu,\alpha_i\ra]$, so the basis from \refc{freeness-over-polynomials} yields a basis for $Y_\mu$ over $\BC$.  This completes the proof of \reft{th: PBW antidominant}.
\end{proof}

\sec{coproduct}{Coproduct}
We continue with $ \fg $ a simply-laced semisimple Lie algebra.

In this section we will describe a family of coproducts for shifted Yangians.  For any splitting $\mu = \mu_1 + \mu_2$, in \reft{generalcoproduct} we will establish the existence of a homomorphism
\eq{eq: coproducts begin}
\Delta_{\mu_1, \mu_2}\colon Y_\mu \longrightarrow Y_{\mu_1} \otimes Y_{\mu_2}
\end{equation}
This generalizes the coproduct for the ordinary Yangian $Y = Y_0$.

\ssec{antidominant coproduct}{A new presentation}
Let $\mu$ be an antidominant coweight. We will follow \cite{L} and define another presentation for $Y_\mu$.

 Fix a decomposition $\mu=\mu_1 + \mu_2$ where the $\mu_i$'s are antidominant coweights.

Denote by $Y_{\mu_1,\mu_2}$ the algebra generated by: $S_i^{(-\langle \mu ,\alpha_i\rangle +1)}$, $S_i^{(-\langle \mu ,\alpha_i\rangle +2)}$, $E_i^{(r)}$ ($1\leq r \leq -\langle \mu_1,\alpha_i \rangle +2$), $F_i^{(r)}$ ($1\leq r \leq -\langle \mu_2,\alpha_i \rangle +2$) for all $i\in I$, with the following relations:
\begin{align}
\label{SS} [S_i^{(k)},S_j^{(l)}]&=0;\\
\label{S1E} [S_i^{(-\langle \mu,\alpha_i\rangle +1)}, E_j^{(r)}]&=(\alpha_i \cdot \alpha_j)   E_j^{(r)}, \hskip1ex 1\leq r \leq \langle \mu_1 , \alpha_j\rangle +1; \\
\label{S1F} [S_i^{(-\langle \mu,\alpha_i\rangle +1)}, F_j^{(r)}]&=-(\alpha_i \cdot \alpha_j)   F_j^{(r)},  \hskip1ex 1\leq r \leq \langle \mu_2 , \alpha_j\rangle +1; \\
\label{S2E} [S_i^{(-\langle \mu, \alpha_i \rangle +2)} ,E_j^{(r)}]&=(\alpha_i \cdot \alpha_j)  E_j^{(r+1)},  \hskip1ex 1\leq r \leq \langle \mu_1 , \alpha_j\rangle +1;\\
\label{S2F} [S_i^{(-\langle \mu, \alpha_i \rangle +2)} ,F_j^{(r)}]&=-(\alpha_i \cdot \alpha_j)  F_j^{(r+1)}, \hskip1ex 1\leq r \leq \langle \mu_2 , \alpha_j\rangle +1;\\
\label{EF} [E_i^{(r)},F_j^{(s)}]=&
\begin{cases}
      0 & i\neq j \\
      0& i=j, r+s < -\langle \mu,\alpha_i \rangle +1 \\
       1& i=j, r+s= -\langle \mu,\alpha_i \rangle +1	\\
       S_i^{(-\langle \mu,\alpha_i \rangle +1)} & i=j, r+s=-\langle \mu,\alpha_i \rangle +2 \\
       S_i^{(-\langle \mu,\alpha_i \rangle +2)} + \frac{1}{2}\big(S_i^{(-\langle \mu,\alpha_i \rangle +1)} \big)^2 & i=j, r+s=-\langle \mu,\alpha_i \rangle +3
   \end{cases}\\
\label{EE} [E_i^{(r+1)},E_j^{(s)}]&=[E_i^{(r)},E_j^{(s+1)}]+\frac{\alpha_i \cdot \alpha_j }{2}(E_i^{(r)}E_j^{(s)} +E_j^{(s)}E_i^{(r)});\\
\label{FF} [F_i^{(r+1)},F_j^{(s)}]&=[F_i^{(r)},F_j^{(s+1)}]-\frac{\alpha_i \cdot \alpha_j }{2}(F_i^{(r)}F_j^{(s)} +F_j^{(s)}F_i^{(r)} );\\
\label{Eserre}\operatorname{ad}(E_i^{(1)}&)^{1-(\alpha_i \cdot \alpha_j) }(E_j^{(1)})=0;\\
\label{Fserre} \operatorname{ad}(F_i^{(1)}&)^{1-(\alpha_i \cdot \alpha_j) }(F_j^{(1)})=0;\\
\label{FSE}  \big[S_i^{(-\langle \mu,\alpha_i\rangle +2)},&[E_i^{(-\langle \mu_1,\alpha_i\rangle +2)},F_i^{(-\langle \mu_2,\alpha_i\rangle +2)}]\big]=0.
\end{align}

For $r\geq 2$ and $s\geq 1$, set
\begin{align*}
E_i^{(r)}&=\frac{1}{2}[S_i^{(-\langle \mu,\alpha_i \rangle +2 )},E_i^{(r-1)}];\\
F_i^{(r)}&=-\frac{1}{2}[S_i^{(-\langle \mu,\alpha_i \rangle +2 )},F_i^{(r-1)}];\\
H_i^{(s)}&= [E_i^{(1)},F_i^{(s)}].
\end{align*}

\begin{Rem}
Note that $H_i^{(s)}= 0$ if $s<-\langle \mu,\alpha_i\rangle$ and $H_i^{(-\langle \mu,\alpha_i\rangle )}=1$.
\end{Rem}

Next, we have the following theorem, whose proof is almost exactly the same as in \cite{L}.

\begin{Thm}\label{newpresentation}
There exists an isomorphism $Y_\mu \longrightarrow Y_{\mu_1,\mu_2}$ of unital associative algebras given by
\begin{align*}
E_i^{(r)} \mapsto E_i^{(r)},F_i^{(r)} \mapsto F_i^{(r)}, H_i^{(s)} \mapsto H_i^{(s)},
\end{align*}
for $r\geq 1$ and $s\geq -\langle \mu,\alpha_i\rangle +1$.
\end{Thm}

\ssec{Yangian coproduct}{The coproduct for the ordinary Yangian}
Recall the following theorem going back to Drinfeld \cite{D}:

\th{thm: coproduct for ordinary yangian}
There is a homomorphism $\Delta\colon Y \longrightarrow Y\otimes Y$, defined on the generators (see Lemma \ref{4generators}) by
\begin{align*}
\Delta(X_i^{(1)} ) & = X_i^{(1)} \otimes 1 + 1 \otimes X_i^{(1)}, \qquad \text{for }X = E, F, S, \\
\Delta(S_i^{(2)}) & = S_i^{(2)} \otimes 1  + 1 \otimes S_i^{(2)} - \sum_{\gamma> 0 }  \langle \alpha_i, \gamma\rangle F_\gamma^{(1)} \otimes E_\gamma^{(1)}
\end{align*}
\eth

This formula for the coproduct was given without a proof
in~\cite[(2.8)--(2.11)]{kt96}.
The proof is given in a recent paper of Guay-Nakajima-Wendlandt \cite[Theorem~4.1]{GN}.

\ssec{Definition of coproduct}{The coproduct in the antidominant case}
Let $\mu, \mu_1, \mu_2$ be antidominant coweights with $\mu = \mu_1 + \mu_2$. We wish to define a homomorphism $\Delta_{\mu_1,\mu_2}: Y_\mu \longrightarrow Y_{\mu_1} \otimes Y_{\mu_2}$ (we will denote $\Delta = \Delta_{\mu_1,\mu_2}$ when the algebras involved are clear).  To do so, we define it on generators as follows:
\begin{align*}
\Delta(E_i^{(r)})&= E_i^{(r)}\otimes 1, \hskip1ex 1\leq r \leq -\langle \mu_1,\alpha_i \rangle ;\\
\Delta(E_i^{(-\langle \mu_1,\alpha_i \rangle+1)})&=E_i^{(-\langle \mu_1 ,\alpha_i \rangle +1)}\otimes 1 + 1\otimes E_i^{(1)};\\
\Delta(E_i^{(-\langle \mu_1,\alpha_i \rangle+2)})&=E_i^{(-\langle \mu_1,\alpha_i\rangle+2)}\otimes1+1\otimes E_i^{(2)}+S_i^{(-\langle \mu_1,\alpha_i\rangle+1)}\otimes E_i^{(1)}\\
&\hskip2ex -\sum\limits_{\gamma>0}F_\gamma^{(1)}\otimes[E_i^{(1)},E_\gamma^{(1)}];\\
\Delta(F_i^{(r)})&= 1\otimes F_i^{(r)}, \hskip1ex 1\leq r \leq -\langle \mu_2,\alpha_i \rangle; \\
\Delta(F_i^{(-\langle \mu_2,\alpha_i \rangle+1)})&=1\otimes F_i^{(-\langle \mu_2 ,\alpha_i \rangle +1)}+ F_i^{(1)}\otimes 1;\\
\Delta(F_i^{(-\langle \mu_2,\alpha_i \rangle+2)})&=1\otimes F_i^{(-\langle \mu_2,\alpha_i\rangle+2)}+ F_i^{(2)}\otimes 1 +F_i^{(1)}\otimes S_i^{(-\langle \mu_2,\alpha_i\rangle+1)}\\
&\hskip2ex +\sum\limits_{\gamma>0}[F_i^{(1)},F_\gamma^{(1)}]\otimes E_\gamma^{(1)};\\
\Delta(S_i^{(-\langle \mu,\alpha_i\rangle +1 )})&=S_i^{(-\langle \mu_1,\alpha_i\rangle +1 )}\otimes 1 + 1\otimes S_i^{(-\langle \mu_2,\alpha_i\rangle +1 )};\\
\Delta(S_i^{(-\langle \mu,\alpha_i\rangle +2 )})&= S_i^{(-\langle \mu_1,\alpha_i\rangle +2 )}\otimes 1 + 1\otimes S_i^{(-\langle \mu_2,\alpha_i\rangle +2 )} - \sum\limits_{\gamma>0}\langle \alpha_i ,\gamma \rangle F_\gamma^{(1)} \otimes E_\gamma^{(1)}.
\end{align*}

\begin{Rem}
\label{rem: coproduct existence}
When $\mu=\mu_1 = \mu_2= 0$, it is not hard to see that $\Delta_{0,0}$ agrees with the coproduct from \reft{thm: coproduct for ordinary yangian}, and hence is well-defined.

Recall that there are shift maps $\iota_{0,\mu_1,0}$ and $\iota_{0,0,\mu_2}$, by \refp{shiftmaps}. It is not hard to see that, for $k=1,2$,
\begin{align*}
\Delta(S_i^{(-\langle \mu,\alpha_i\rangle +k)})=(\iota_{0,\mu_1,0}\otimes \iota_{0,0,\mu_2})\Delta_{0,0}(S_i^{(k)}),\\
\Delta(E_i^{(-\langle \mu_1,\alpha_i\rangle +k)})=(\iota_{0,\mu_1,0}\otimes \iota_{0,0,\mu_2})\Delta_{0,0}(E_i^{(k)}),\\
\Delta(F_i^{(-\langle \mu_2,\alpha_i\rangle +k)})=(\iota_{0,\mu_1,0}\otimes \iota_{0,0,\mu_2})\Delta_{0,0}(F_i^{(k)}).
\end{align*}
\end{Rem}

\begin{Thm}
\label{thm: coproduct existence}
$\Delta\colon Y_\mu \longrightarrow Y_{\mu_1} \otimes Y_{\mu_2}$ is a well-defined map.
\end{Thm}
\begin{proof}
We have to check that $\Delta$ preserves the defining relations.  By Theorem \ref{newpresentation} it suffices to check the relations \eqref{SS} -- \eqref{FSE}.

First, we check relation \eqref{SS}. For $1\leq k,l\leq 2$,
\begin{align*}
&[\Delta(S_i^{(-\langle \mu,\alpha_i \rangle + k)}),\Delta(S_j^{(-\langle \mu,\alpha_j \rangle +l)})]=\\
&= [(\iota_{0,\mu_1,0}\otimes \iota_{0,0,\mu_2})\Delta_{0,0}(S_i^{(k)}),(\iota_{0,\mu_1,0}\otimes \iota_{0,0,\mu_2})\Delta_{0,0}(S_j^{(l)})]=0.
\end{align*}
\vskip1ex
We check relation \eqref{S1E}. For $1\leq r\leq -\langle \mu_1 ,\alpha_j \rangle$,
\begin{align*}
[\Delta(S_i^{(-\langle \mu,\alpha_i\rangle +1)}), \Delta(E_j^{(r)})]&= [S_i^{(-\langle \mu_1,\alpha_i\rangle +1)},E_j^{(r)}]\otimes 1 = (\alpha_i\cdot \alpha_j)\Delta(E_j^{(r)}).
\end{align*}
For $r=-\langle \mu_1, \alpha_j\rangle + 1$,
\begin{align*}
[\Delta(S_i^{(-\langle \mu,\alpha_i \rangle + 1)}),\Delta(E_j^{(-\langle \mu_1,\alpha_j\rangle +1)}]
&= (\iota_{0,\mu_1,0}\otimes \iota_{0,0,\mu_2})\Delta_{0,0}([S_i^{(1)},E_j^{(1)}]) \\
&=  (\alpha_i\cdot \alpha_j)(\iota_{0,\mu_1,0}\otimes \iota_{0,0,\mu_2})\Delta_{0,0}(E_j^{(1)})\\
&=  (\alpha_i\cdot \alpha_j ) \Delta(E_j^{(-\langle \mu_1,\alpha_j\rangle +1)}).
\end{align*}
\vskip1ex
The proof for relation \eqref{S1F} is similar to that of relation \eqref{S1E}.
\vskip1ex
We check relation \eqref{S2E}. For $1\leq r\leq -\langle \mu_1 ,\alpha_j \rangle$,
\begin{align*}
[\Delta(S_i^{(-\langle \mu,\alpha_i\rangle +2)}), \Delta(E_j^{(r)})]
&=[S_i^{(-\langle \mu_1,\alpha_i\rangle +2)},E_j^{(r)}]\otimes 1+\sum \limits_{\gamma >0}\langle \alpha_i ,\gamma \rangle [E_j^{(r)},F_\gamma^{(1)}]\otimes E_\gamma^{(1)}\\
&=(\alpha_i\cdot \alpha_j)E_j^{(r+1)}\otimes 1 +\sum \limits_{\gamma >0}\langle \alpha_i ,\gamma \rangle [E_j^{(r)},F_\gamma^{(1)}]\otimes E_\gamma^{(1)}.
\end{align*}
Note that if $r<-\langle \mu_1,\alpha_i \rangle$, then $[E_j^{(r)},F_l^{(1)}]=0$ for all $l$. Then, by induction,  $[E_j^{(r)},F_\gamma^{(1)}]=0$ for all $\gamma>0$. The result follows in this case. If $r=-\langle \mu_1,\alpha_i \rangle$, then  $[E_j^{(r)},F_i^{(1)}]=\delta_{ij} 1$. Then, by induction,  $[E_j^{(r)},F_\gamma^{(1)}]=0$ for all $\gamma$ of height greater than or equal to 2. The second summand becomes $(\alpha_i\cdot \alpha_j)1\otimes E_j^{(1)}$. Hence, the result follows.

For $r=-\langle \mu_1, \alpha_j\rangle + 1$,
\begin{align*}
[\Delta(S_i^{(-\langle \mu,\alpha_i \rangle + 2)}),\Delta(E_j^{(-\langle \mu_1,\alpha_j\rangle +1)})]
&= (\iota_{0,\mu_1,0}\otimes \iota_{0,0,\mu_2})\Delta_{0,0}([S_i^{(2)},E_j^{(1)}]) \\
&=  (\alpha_i\cdot \alpha_j)(\iota_{0,\mu_1,0}\otimes \iota_{0,0,\mu_2})\Delta_{0,0}(E_j^{(2)})\\
&=  (\alpha_i\cdot \alpha_j ) \Delta(E_j^{(-\langle \mu_1,\alpha_j\rangle +2)}).
\end{align*}
\vskip1ex
Similarly, $\Delta$ preserves relation \eqref{S2F}.
\vskip1ex
Next, we check relation \eqref{EF}. If $1\leq r \leq -\langle \mu_1 ,\alpha_i \rangle $ and $1\leq s \leq -\langle \mu_2 ,\alpha_j \rangle $, then
\begin{align*}
[\Delta(E_i^{(r)}),\Delta(F_j^{(s)})] =[E_i^{(r)}\otimes 1,1\otimes F_j^{(s)}]=0.
\end{align*}

For $r=-\langle \mu_1,\alpha_i \rangle +1$ and $1\leq s \leq -\langle \mu_2 ,\alpha_j \rangle $,
\begin{align*}
[\Delta(E_i^{(-\langle \mu_1,\alpha_i \rangle +1)}),\Delta(F_j^{(s)})]
&=1\otimes [E_i^{(1)},F_j^{(s)}]=\delta_{ij}1\otimes H_i^{(s)}.
\end{align*}
The result follows for this case.

The case where $r\leq -\langle \mu_1,\alpha_i \rangle $ and $s =-\langle \mu_2 ,\alpha_j \rangle +1 $ is similar.

Consider the case where $r=-\langle \mu_1,\alpha_i \rangle +2$ and $1\leq s \leq -\langle \mu_2 ,\alpha_j \rangle $,
\begin{align*}
&[\Delta(E_i^{(-\langle \mu_1,\alpha_i \rangle +2)}),\Delta(F_j^{(s)})]=\\
&= 1\otimes [E_i^{(2)},F_j^{(s)}]+S_i^{(-\langle \mu_1,\alpha_i\rangle+1)}\otimes [E_i^{(1)},F_j^{(s)}]- \sum_{\gamma >0}F_\gamma^{(1)}\otimes [[E_i^{(1)},E_\gamma^{(1)}],F_j^{(s)}]\\
&=\delta_{ij} 1\otimes H_i^{(s+1)}+\delta_{ij} S_i^{(-\langle \mu_1,\alpha_i\rangle+1)}\otimes H_i^{(s)}- \sum_{\gamma >0}F_\gamma^{(1)}\otimes [[E_i^{(1)},E_\gamma^{(1)}],F_j^{(s)}].
\end{align*}

Note that
\begin{align*}
[[E_i^{(1)},E_\gamma^{(1)}],F_j^{(s)}]= [E_i^{(1)},[E_\gamma^{(1)},F_j^{(s)}]] + [E_\gamma^{(1)},[F_j^{(s)},E_i^{(1)}]].
\end{align*}
 Since $s\leq -\langle \mu_2 ,\alpha_j\rangle $, by induction, $[E_\gamma^{(1)}, F_j^{(s)}]\in \mathbb{C}  1$. Hence, $[E_i^{(1)},[E_\gamma^{(1)}, F_j^{(s)}]]=0$.
Again, since $s\leq -\langle \mu_2,\alpha_j\rangle$, $[F_j^{(s)},E_i^{(1)}]=\delta_{ij} H_j^{(s)}\in \mathbb{C}  1$. So, $[E_\gamma^{(1)},[F_j^{(s)},E_i^{(1)}]]=0$. Hence, the last sum is 0. Moreover, it is straightforward to check that the first two summands are consistent with the relation.

The case where $1\leq r \leq -\langle \mu_1 ,\alpha_i \rangle$ and $s=-\langle \mu_2,\alpha_j \rangle +2$ is similar.

Next, for $1\leq k,l\leq 2$ not both equal to 2, we have that
\begin{align*}
[\Delta(E_i^{(-\langle \mu_1,\alpha_i \rangle + k)}),\Delta(F_j^{(-\langle \mu_2,\alpha_j\rangle +l)}]
&= (\iota_{0,\mu_1,0}\otimes \iota_{0,0,\mu_2})\Delta_{0,0}([E_i^{(k)},F_j^{(l)}])\\
&= \delta_{ij}(\iota_{0,\mu_1,0}\otimes \iota_{0,0,\mu_2})\Delta_{0,0}(H_i^{(k+l-1)})\\
&=  \delta_{ij}\Delta(H_i^{(-\langle \mu,\alpha_i\rangle +k+l-1)}).
\end{align*}

\vskip1ex
Next, we check relation \eqref{EE}.

First, consider the case $1\leq r < -\langle \mu_1 ,\alpha_i \rangle$ and $1\leq s < -\langle \mu_1,\alpha_j \rangle$. Then, we have
\begin{align*}
&[\Delta(E_i^{(r+1)}),\Delta(E_j^{(s)})]=[E_i^{(r+1)},E_j^{(s)}]\otimes 1=\\
&= \Big([E_i^{(r)},E_j^{(s+1)}]+\frac{ \alpha_i\cdot\alpha_j  }{2}(E_i^{(r)}E_j^{(s)} +E_j^{(s)}E_i^{(r)} )\Big) \otimes 1\\
&= [E_i^{(r)}\otimes 1,E_j^{(s+1)}\otimes 1]+\frac{ \alpha_i\cdot\alpha_j  }{2}\big((E_i^{(r)}\otimes 1)(E_j^{(s)}\otimes 1) +(E_j^{(s)}\otimes 1)(E_i^{(r)}\otimes 1) \big).
\end{align*}
Consider the case $1\leq r < -\langle \mu_1 ,\alpha_i \rangle$ and $s=-\langle \mu_1 ,\alpha_j \rangle$.
\begin{align*}
&[\Delta(E_i^{(r+1)}),\Delta(E_j^{(-\langle \mu_1 ,\alpha_j \rangle)})]- [\Delta(E_i^{(r)}),\Delta(E_j^{(-\langle \mu_1 ,\alpha_j \rangle+1)})]\\
&= ([E_i^{(r+1)}, E_j^{(-\langle \mu_1 ,\alpha_j \rangle)}]-  [ E_i^{(r)}, E_j^{(-\langle \mu_1 ,\alpha_j \rangle+1)}])\otimes 1\\
&= \frac{ \alpha_i\cdot\alpha_j  }{2}\big((E_i^{(r)}\otimes 1)(E_j^{(-\langle \mu_1 ,\alpha_j \rangle)}\otimes 1)+ (E_j^{(-\langle \mu_1 ,\alpha_j \rangle)}\otimes 1)(E_j^{(r)}\otimes 1)\big).
\end{align*}
The case where $r=-\langle \mu_1,\alpha_i \rangle$ and $1\leq s <-\langle \mu_1,\alpha_j \rangle$ is similar

Next, consider the case $1\leq r < -\langle \mu_1 ,\alpha_i \rangle$ and $s=-\langle \mu_1 ,\alpha_j \rangle+1$.
\begin{align*}
&[\Delta(E_i^{(r+1)}),\Delta(E_j^{(-\langle \mu_1 ,\alpha_j \rangle)+1})]- [\Delta(E_i^{(r)}),\Delta(E_j^{(-\langle \mu_1 ,\alpha_j \rangle+2)})]\\
&= [E_i^{(r+1)},E_j^{(-\langle \mu_1,\alpha_j \rangle +1)} ]\otimes 1 - [E_i^{(r)},E_j^{(-\langle \mu_1,\alpha_j \rangle +2)}]\otimes 1- [E_i^{(r)}, S_j^{(-\langle \mu_1,\alpha_j \rangle +1)}]\otimes E_j^{(1)} \\
&+ \sum_{\gamma >0} [E_i^{(r)},F_\gamma^{(1)}]\otimes [E_j^{(1)},E_\gamma^{(1)}] \\
&= \frac{ \alpha_i\cdot\alpha_j  }{2}(E_i^{(r)}E_j^{(-\langle \mu_1,\alpha_j \rangle +1)}+ E_j^{(-\langle \mu_1,\alpha_j \rangle +1)}E_i^{(r)})\otimes 1 +  (\alpha_i\cdot\alpha_j)  E_i^{(r)}\otimes E_j^{(1)} \\
&+\sum_{\gamma >0} [E_i^{(r)},F_\gamma^{(1)}]\otimes [E_j^{(1)},E_\gamma^{(1)}].
\end{align*}
Since $r<-\langle \mu_1, \alpha_i \rangle$, by induction, $[E_i^{(r)},F_\gamma^{(1)}]=0$ for all $\gamma>0$. The current case follows.
The proof for $r=-\langle \mu_1,\alpha_i \rangle +1$ and $s<-\langle \mu_1,\alpha_j\rangle $ is similar.

Next, let us look at the case $r= -\langle \mu_1,\alpha_i\rangle$ and $s=-\langle \mu_1,\alpha_j\rangle$.
\begin{align*}
&[\Delta(E_i^{(-\langle \mu_1,\alpha_i\rangle +1)}),\Delta(E_j^{(-\langle \mu_1,\alpha_j\rangle)})]- [\Delta(E_i^{(-\langle \mu_1,\alpha_i\rangle )}),\Delta(E_j^{(-\langle \mu_1,\alpha_j\rangle+1)})]\\
&= ([E_i^{(-\langle \mu_1,\alpha_i\rangle)},E_j^{(-\langle \mu_1,\alpha_j\rangle)}] - [E_i^{(-\langle \mu_1,\alpha_i\rangle)}, E_j^{(-\langle \mu_1,\alpha_j\rangle+1)}])\otimes 1\\
&=\frac{\alpha_i\cdot \alpha_j}{2}\big((E_i^{(-\langle \mu_1,\alpha_i\rangle)}\otimes 1)(E_j^{(-\langle \mu_1,\alpha_j\rangle)}\otimes 1) + (E_j^{(-\langle \mu_1,\alpha_j\rangle)}\otimes 1)(E_i^{(-\langle \mu_1,\alpha_i\rangle)}\otimes 1)\big).
\end{align*}

Next, for $r= -\langle \mu_1,\alpha_i \rangle $ and $s= -\langle \mu_1,\alpha_j \rangle +1$.
\begin{align*}
&[\Delta(E_i^{(-\langle \mu_1,\alpha_i \rangle +1)}),\Delta(E_j^{(-\langle \mu_1 ,\alpha_j \rangle)+1})]- [\Delta(E_i^{(-\langle \mu_1,\alpha_i \rangle )}),\Delta(E_j^{(-\langle \mu_1 ,\alpha_j \rangle+2)})]= \\
&= [E_i^{(-\langle \mu_1,\alpha_i \rangle +1)}, E_j^{(-\langle \mu_1,\alpha_j \rangle +1)}]\otimes 1 + 1\otimes [E_i^{(1)},E_j^{(1)}]-[E_i^{(-\langle \mu_1,\alpha_i \rangle )},E_j^{(-\langle \mu_1,\alpha_j \rangle +2)}]\otimes 1\\
&- [E_i^{(-\langle \mu_1,\alpha_i \rangle )},S_j^{(-\langle \mu_1,\alpha_j \rangle +1)}]\otimes E_j^{(1)} + \sum_{\gamma>0} [E_i^{(-\langle \mu_1,\alpha_i \rangle )}, F_\gamma^{(1)}]\otimes [E_j^{(1)},E_\gamma^{(1)}]\\
&= \frac{ \alpha_i\cdot\alpha_j  }{2}\big(E_i^{(-\langle \mu_1,\alpha_i \rangle)}E_j^{-\langle \mu_1,\alpha_j \rangle +1)}\otimes 1 +E_j^{(-\langle \mu_1,\alpha_j \rangle +1)}E_i^{(-\langle \mu_1,\alpha_i \rangle +1)}\otimes 1 \\
&+ 2E_i^{(-\langle \mu_1,\alpha_i \rangle )}\otimes E_j^{(1)}\big)
+ 1\otimes [E_i^{(1)},E_j^{(1)}] + \sum_{\gamma>0} [E_i^{(-\langle \mu_1,\alpha_i \rangle )}, F_\gamma^{(1)}]\otimes [E_j^{(1)},E_\gamma^{(1)}].
\end{align*}
Note that $[E_i^{(-\langle \mu_1,\alpha_i \rangle )}, F_l^{(1)}]\in \mathbb{C} 1$. So, if $\gamma$ is of height greater than or equal to 2, then $[E_i^{(-\langle \mu_1 ,\alpha_i\rangle )},F_\gamma^{(1)}]=0$ by induction. Hence, the only term that survives in the last summand is $1\otimes [E_j^{(1)},E_i^{(1)}]$ and we are done.
The case $r= -\langle \mu_1,\alpha_i \rangle +1$ and $s= -\langle \mu_1,\alpha_j \rangle $ is totally analogous.

Lastly, consider the case $r= -\langle \mu_1,\alpha_i \rangle +1$ and $s= -\langle \mu_1,\alpha_j \rangle +1$.
\begin{align*}
&[\Delta(E_i^{(-\langle \mu_1,\alpha_i \rangle +2)}),\Delta(E_j^{(-\langle \mu_1 ,\alpha_j \rangle)+1})]- [\Delta(E_i^{(-\langle \mu_1,\alpha_i \rangle +1)}),\Delta(E_j^{(-\langle \mu_1 ,\alpha_j \rangle+2)})]= \\
&= (\iota_{0,\mu_1,0}\otimes \iota_{0,0,\mu_2})\Delta_{0,0} \big( [E_i^{(2)},E_j^{(1)}]-[E_i^{(1)}E_j^{(2)}]\big)\\
&=  \frac{\alpha_i\cdot \alpha_j}{2}(\iota_{0,\mu_1,0}\otimes \iota_{0,0,\mu_2})\Delta_{0,0}\big(E_i^{(1)}E_j^{(1)}+E_j^{(1)}E_i^{(1)}\big)\\
&=  \frac{\alpha_i\cdot \alpha_j}{2}\big( \Delta(E_i^{(-\langle \mu_1,\alpha_i \rangle +1)})\Delta(E_j^{(-\langle \mu_1,\alpha_j \rangle +1)})+\Delta(E_j^{(-\langle \mu_1,\alpha_j \rangle +1)})\Delta(E_i^{(-\langle \mu_1,\alpha_i \rangle +1)})\big).
\end{align*}

Relation \eqref{FF} can be checked in the same fashion.
\vskip1ex
We now check relation \eqref{Eserre}. Set $N=1-\alpha_i\cdot \alpha_j$.
First, if $1\leq -\langle \mu_1, \alpha_i\rangle$ and $1\leq -\langle \mu_1,\alpha_j \rangle$, then
\begin{align*}
\operatorname{ad} (\Delta(E_i^{(1)}))^{N}(\Delta(E_j^{(1)}))&= \operatorname{ad}(E_i^{(1)}\otimes 1)^{N}(E_j^{(1)}\otimes 1)= \big((\operatorname{ad} E_i^{(1)})^{N}(E_j^{(1)})\big)\otimes 1= 0.
\end{align*}

For $1\leq -\langle \mu_1, \alpha_i\rangle$ and $1=-\langle \mu_1,\alpha_j \rangle +1$.
\begin{align*}
\operatorname{ad} (\Delta(E_i^{(1)}))^{N}(\Delta(E_j^{(1)}))&= \operatorname{ad}(E_i^{(1)}\otimes 1)^{N}(E_j^{(1)}\otimes 1+ 1\otimes E_j^{(1)})\\
&= \big((\operatorname{ad} E_i^{(1)})^{N}(E_j^{(1)})\big)\otimes 1= 0,
\end{align*}
since $E_i^{(1)}\otimes 1$ commutes with $1\otimes E_j^{(1)}$.

Next, suppose $1=-\langle \mu_1,\alpha_i \rangle +1$.
Since $[E_i^{(1)}\otimes 1,1\otimes E_i^{(1)}]=0$,
\begin{align*}
(\operatorname{ad}(\Delta(E_i^{(1)})))^N&=(\operatorname{ad}(E_i^{(1)}\otimes 1)+\operatorname{ad}(1\otimes E_i^{(1)}))^N\\
&=\sum_{l=0}^{N}\binom{N}{l}\operatorname{ad}(E_i^{(1)}\otimes 1)^i \operatorname{ad} (1\otimes E_i^{(1)})^{N-i}.
\end{align*}
Now, if $1\leq -\langle \mu_1,\alpha_j\rangle$, then
\begin{align*}
(\operatorname{ad}(\Delta(E_i^{(1)})))^N(\Delta(E_j^{(1)})
&= \sum_{l=0}^{N}\binom{N}{l}\operatorname{ad}(E_i^{(1)}\otimes 1)^i \operatorname{ad} (1\otimes E_i^{(1)})^{N-i}(E_j^{(1)}\otimes 1)\\
&= \operatorname{ad}(E_i^{(1)})^N(E_j^{(1)})\otimes 1=0.
\end{align*}
If $1=-\langle \mu_1,\alpha_j\rangle +1$, then
\begin{align*}
\operatorname{ad} (\Delta(E_i^{(1)}))^{N}(\Delta(E_j^{(1)}))
&= \sum_{l=0}^{N}\binom{N}{l}\operatorname{ad}(E_i^{(1)}\otimes 1)^i \operatorname{ad} (1\otimes E_i^{(1)})^{N-i}(E_j^{(1)}\otimes 1+1\otimes E_j^{(1)})\\
&= \operatorname{ad}(E_i^{(1)})^N(E_j^{(1)})\otimes 1 + 1\otimes \operatorname{ad}(E_i^{(1)})^N(E_j^{(1)})  = 0.
\end{align*}
\vskip1ex
The proof for \eqref{Fserre} is similar to that of \eqref{Eserre}.
\vskip1ex
Finally, we check relation \eqref{FSE}.
\begin{align*}
 &\big[\Delta(S_i^{(-\langle \mu,\alpha_i\rangle +2)}),[\Delta(E_i^{(-\langle \mu_1,\alpha_i\rangle +2)}),\Delta(F_i^{(-\langle \mu_2,\alpha_i\rangle +2)})]\big]\\
&= (\iota_{0,\mu_1,0}\otimes \iota_{0,0,\mu_2})\Delta_{0,0}\big([S_i^{(2)},[E_i^{(2)},F_i^{(2)}]\big)=0.
\end{align*}
This proves that $\Delta$ is well-defined.
\end{proof}

By Lemma \ref{4generators}, we have the following:

\lem{4generatorscoproduct}
The coproduct $\Delta\colon Y_\mu \longrightarrow Y_{\mu_1}\otimes Y_{\mu_2}$ is uniquely determined by
\begin{align*}
\Delta(E_i^{(1)})&=E_i^{(1)}\otimes 1 + \delta_{\langle \mu_1,\alpha_i\rangle ,0}1\otimes E_i^{(1)};\\
\Delta(F_i^{(1)})&=\delta_{\langle \mu_2 ,\alpha_i\rangle ,0}F_i^{(1)}\otimes 1 + 1\otimes F_i^{(1)};\\
\Delta(S_i^{(-\langle \mu,\alpha_i\rangle +1 )})&=S_i^{(-\langle \mu_1,\alpha_i\rangle +1 )}\otimes 1 + 1\otimes S_i^{(-\langle \mu_2,\alpha_i\rangle +1 )};\\
\Delta(S_i^{(-\langle \mu,\alpha_i\rangle +2 )})&= S_i^{(-\langle \mu_1,\alpha_i\rangle +2 )}\otimes 1 + 1\otimes S_i^{(-\langle \mu_2,\alpha_i\rangle +2 )} - \sum\limits_{\gamma>0}\langle \alpha_i ,\gamma \rangle F_\gamma^{(1)} \otimes E_\gamma^{(1)}.
\end{align*}
\elem

\begin{Prop} \label{coassociative}
Let $\mu=\mu_1+\mu_2+\mu_3$ where the $\mu_i$'s are antidominant coweights. Then, we have the following commutative diagram
\[
\xymatrix{
Y_{\mu} \ar[rr]^{\Delta_{\mu_1,\mu_2+\mu_3}} \ar[d]_{\Delta_{\mu_1+\mu_2,\mu_3}}&& Y_{\mu_1}\otimes Y_{\mu_2+\mu_3} \ar[d]^{1\otimes \Delta_{\mu_2,\mu_3}}\\
Y_{\mu_1+\mu_2}\otimes Y_{\mu_3} \ar[rr]_{\Delta_{\mu_1,\mu_2}\otimes 1}&& Y_{\mu_1}\otimes Y_{\mu_2}\otimes Y_{\mu_3}
}
\]
\end{Prop}
\begin{proof}
By \refl{4generatorscoproduct}, it is enough to check for $S_i^{(-\langle \mu,\alpha_i\rangle +k)}$ ($k=1,2$), $E_i^{(1)}$ and $F_i^{(1)}$.
\begin{align*}
(1\otimes \Delta_{\mu_2,\mu_3})\Delta_{\mu_1,\mu_2+\mu_3}(E_i^{(1)})
&= E_i^{(1)}\otimes 1 \otimes 1 + \delta_{\langle \mu_1,\alpha_i\rangle ,0}1\otimes E_i^{(1)}\otimes 1 \\
&+ \delta_{\langle \mu_1,\alpha_i\rangle,0}\delta_{\langle \mu_2,\alpha_i\rangle ,0}1\otimes 1\otimes E_i^{(1)},\\
(\Delta_{\mu_1,\mu_2}\otimes 1)\Delta_{\mu_1+\mu_2,\mu_3}(E_i^{(1)})
&= E_i^{(1)}\otimes 1\otimes 1 + \delta_{\langle \mu_1, \alpha_i\rangle ,0}1\otimes E_i^{(1)}\otimes 1 \\
&+ \delta_{\langle \mu_1+\mu_2,\alpha_i\rangle}1\otimes 1 \otimes E_i^{(1)}.
\end{align*}
The result follows for $E_i^{(1)}$ since $\delta_{\langle \mu_1+\mu_2,\alpha_i\rangle ,0}=\delta_{\langle \mu_1,\alpha_i\rangle,0}\delta_{\langle \mu_2,\alpha_i\rangle,0}$. The computation for $F_i^{(1)}$ is totally analogous.
The computation $S_i^{(-\langle \mu,\alpha_i\rangle+1)}$ is straightforward.

Finally, we have that
\begin{align*}
&(1\otimes \Delta_{\mu_2,\mu_3})\Delta_{\mu_1,\mu_2+\mu_3}(S_i^{(-\langle \mu, \alpha_i\rangle+2)})= \\
&=1\otimes S_i^{(-\langle \mu_2, \alpha_i\rangle +2)}\otimes 1 + 1\otimes 1\otimes S_i^{(-\langle \mu_3, \alpha_i\rangle +2)} - \sum_{\beta>0}\langle \alpha_i,\beta\rangle 1\otimes F_\beta^{(1)}\otimes E_\beta^{(1)})\\
&+S_i^{(-\langle \mu_1,\alpha_i\rangle +2)}\otimes 1 \otimes 1 - \sum_{\gamma>0} \langle \alpha_i,\gamma \rangle F_\gamma^{(1)}\otimes \Delta_{\mu_2,\mu_3}(E_\gamma^{(1)}),\\
&(\Delta_{\mu_1,\mu_2}\otimes 1)\Delta_{\mu_1+\mu_2,\mu_3}(S_i^{(-\langle \mu,\alpha_i\rangle +2)})= \\
&= S_i^{(-\langle \mu_1,\alpha_i\rangle +2)}\otimes 1 \otimes 1 + 1\otimes S_i^{(-\langle \mu_2,\alpha_i\rangle +2)}\otimes 1  -\sum_{\beta>0}\langle \alpha_i,\beta \rangle F_\beta^{(1)}\otimes E_\beta^{(1)}\otimes 1 \\
&+ 1\otimes 1 \otimes S_i^{(-\langle \mu_3,\alpha_i\rangle +2)} - \sum_{\gamma>0} \langle \alpha_i,\gamma \rangle \Delta_{\mu_1,\mu_2}(F_\gamma^{(1)})\otimes E_\gamma^{(1)}.
\end{align*}

For a positive root $\gamma= \sum_in_i\alpha_i$, by a simply induction, we can show that $\Delta_{\mu_2,\mu_3}(E_\gamma^{(1)})= E_\gamma^{(1)}\otimes 1 + C_\gamma 1\otimes E_\gamma^{(1)}$ and that $\Delta_{\mu_1,\mu_2}(F_\gamma^{(1)})= 1\otimes F_\gamma^{(1)} + C_\gamma F_\gamma^{(1)}\otimes 1$ where $C_\gamma=\prod_i\delta_{\langle \mu_2,\alpha_i\rangle ,0}^{n_i}$. The result follows.
\end{proof}

\subsection{The coproduct in the general case}
\th{generalcoproduct}
Let $\mu=\mu_1+\mu_2$ where $\mu,\mu_1,\mu_2$ are arbitrary coweights. There exists a coproduct $\Delta_{\mu_1,\mu_2}\colon Y_\mu \longrightarrow Y_{\mu_1}\otimes Y_{\mu_2}$ such that, for all antidominant coweights $\eta_1,\eta_2$, the following diagram is commutative
\[
\xymatrix{
Y_\mu  \ar[d]_{\iota_{\mu,\eta_1,\eta_2}} \ar[rrr]^{\Delta_{\mu_1,\mu_2}} &&& Y_{\mu_1}\otimes Y_{\mu_2} \ar[d]^{(\iota_{\mu_1,\eta_1,0})\otimes (\iota_{\mu_2,0,\eta_2})}\\
Y_{\mu+ \eta_1+\eta_2} \ar[rrr]_{\Delta_{\mu_1+\eta_1,\mu_2+\eta_2}} &&& Y_{\mu_1+\eta_1}\otimes Y_{\mu_2+\eta_2}
}
\]
\eth
\begin{proof}
First, we need to define the map $\Delta_{\mu_1,\mu_2}$. Let $\eta_1, \eta_2$ be antidominant coweights such that $\mu_1 +\eta_1$ and $\mu_2+\eta_2$ are also antidominant. We see that $\mu+\eta_1+\eta_2$ is also antidominant.

Consider the diagram
\[
\xymatrix{
Y_\mu  \ar[d]_{\iota_{\mu,\eta_1,\eta_2}} &&& Y_{\mu_1}\otimes Y_{\mu_2} \ar[d]^{(\iota_{\mu_1,\eta_1,0})\otimes (\iota_{\mu_2,0,\eta_2})}\\
Y_{\mu+ \eta_1+\eta_2} \ar[rrr]_{\Delta=\Delta_{\mu_1+\eta_1,\mu_2+\eta_2}} &&& Y_{\mu_1+\eta_1}\otimes Y_{\mu_2+\eta_2}
}
\]

In order to define $\Delta_{\mu_1,\mu_2}$, we need to show that
\[
\Delta(\iota_{\mu,\eta_1,\eta_2}(Y_{\mu}))\subseteq (\iota_{\mu_1,\eta_1,0}\otimes \iota_{\mu_2,0,\eta_2})(Y_{\mu_1}\otimes Y_{\mu_2}).
\]

Note that $Y_{\mu_1+\eta_1}^\leq \otimes Y_{\mu_2+\eta_2}^\geq \subseteq \iota_{\mu_1,\eta_1,0}\otimes \iota_{\mu_2,0,\eta_2}(Y_{\mu_1}\otimes Y_{\mu_2})$.
\vskip2ex
First, for $r\geq 1$, we claim that
\begin{align*}
\Delta(E_i^{(r)})\in E_i^{(r)}\otimes 1 +Y_{\mu_1+\eta_1}^\leq \otimes Y_{\mu_2+\eta_2}^>, \\
\Delta(F_i^{(r)})\in1\otimes F_i^{(r)} + Y_{\mu_1+\eta_1}^<\otimes Y_{\mu_2+\eta_2}^\geq.
\end{align*}

We prove the claim for $E$, the proof for $F$ is similar. We proceed by induction.

If $1\leq -\langle \mu_1+\eta_1,\alpha_i\rangle$, then it is clear since $\Delta(E_i^{(1)})=E_i^{(1)}\otimes 1$.

If $0=\langle \mu_1+\eta_1,\alpha_i\rangle$,  then it is also clear since $\Delta(E_i^{(1)})= E_i^{(1)}\otimes 1 +1\otimes E_i^{(1)}$ and since $1\otimes E_i^{(1)}\in Y_{\mu_1+\eta_1}^\leq\otimes Y_{\mu_2+\eta_2}^>$.

The induction step follows from the fact that $\Delta$ is a homomorphism and the fact that $[S_i^{(-\langle \mu+\eta_1+\eta_2,\alpha_i\rangle+2)},E_i^{(r)}]=2E_i^{(r+1)}$. This proves the claim.
\vskip2ex

Note that $\iota_{\mu,\eta_1,\eta_2}(Y_\mu)$ is generated by $E_i^{(r)} (r>-\langle \eta_1,\alpha_i\rangle )$, $F_i^{(s)} (s>-\langle \eta_2,\alpha_i\rangle )$ and $H_i^{(t)} (t>-\langle \mu+\eta_1+\eta_2,\alpha_i\rangle)$.

Applying the claim for $r>-\langle \eta_1,\alpha_i\rangle$, we get $\Delta(E_i^{(r)}) \in (\iota_{\mu_1,\eta_1,0}\otimes \iota_{\mu_2,0,\eta_2})(Y_{\mu_1}\otimes Y_{\mu_2})$ since $E_i^{(r)}\otimes 1 = (\iota_{\mu_1,\eta_1,0}\otimes \iota_{\mu_2,0,\eta_2})(E_i^{(r+\langle \eta_1,\alpha_i\rangle)}\otimes 1)$.

Similarly, we obtain $\Delta(F_i^{(r)}) \in (\iota_{\mu_1,\eta_1,0}\otimes \iota_{\mu_2,0,\eta_2})(Y_{\mu_1}\otimes Y_{\mu_2})$
for $s>-\langle \eta_2,\alpha_i\rangle$.

Finally, for $t>-\langle \mu+\eta_1+\eta_2,\alpha_i\rangle$,
\begin{align*}
\Delta(H_i^{(t)})&=[\Delta(E_i^{(t)}),\Delta(F_i^{(1)})]\\
&\in [E_i^{(t)}\otimes 1, Y_{\mu_1+\eta_1}^<\otimes Y_{\mu_2+\eta_2}^\geq] +[Y_{\mu_1+\eta_1}^{\leq}\otimes Y_{\mu_2+\eta_2}^>,1\otimes F_i^{(1)}] \\
&\subseteq Y_{\mu_1+\eta_1}^\leq \otimes Y_{\mu_2+\eta_2}^\geq.
\end{align*}
Therefore, we have a coproduct $\Delta_{\mu_1,\mu_2}:Y_\mu \longrightarrow Y_{\mu_1}\otimes Y_{\mu_2}$.

Next, we show that $\Delta_{\mu_1\mu_2}$ is independent of the choice of $\eta_1,\eta_2$, i.e., for all $\eta_1,\eta_2$ such that $\mu_1+\eta_1,\mu_2+\eta_2$ are antidominant, the diagram in the statement of the theorem is commutative.  To see this, let $\eta_1',\eta_2'$ be another such pair of coweights. Consider the diagram
\[
\xymatrix{
Y_\mu  \ar[d]_{\iota_{\mu,\eta_1,\eta_2}}  &&& Y_{\mu_1}\otimes Y_{\mu_2} \ar[d]^{(\iota_{\mu_1,\eta_1,0})\otimes (\iota_{\mu_2,0,\eta_2})}\\
Y_{\mu+ \eta} \ar[d]_{\iota_{\mu+\eta,\eta_1',\eta_2'}} \ar[rrr]_{\Delta_{\mu_1+\eta_1,\mu_2+\eta_2}} &&& Y_{\mu_1+\eta_1}\otimes Y_{\mu_2+\eta_2} \ar[d]^{(\iota_{\mu_1+\eta_1,\eta_1',0})\otimes (\iota_{\mu_2+\eta_2,0,\eta_2'})}
\\
Y_{\mu+\eta+\eta'} \ar[rrr]_{\Delta_{\mu_1+\eta_1+\eta'_1,\mu_2+\eta_2+\eta'_2}}& & & Y_{\mu_1+\eta_1+\eta_1'} \otimes Y_{\mu_2+\eta_2+\eta_2'}
}
\]
We see that $\iota_{\mu+\eta, \eta_1',\eta_2'}\circ\iota_{\mu,\eta_1,\eta_2}=\iota_{\mu, \eta_1+\eta_1',\eta_2+\eta_2'}$ and $\iota_{\mu_1+\eta_1,\eta_1',0}\circ \iota_{\mu_1,\eta_1,0}=\iota_{\mu,\eta_1+\eta_1',0}$ and $\iota_{\mu_2+\eta_2,0,\eta_2'}\circ \iota_{\mu_2,0,\eta_2}=\iota_{\mu_2,0,\eta_2+\eta_2'}$. Moreover, it is not hard to check on generators that the lower square commutes.

Therefore, the choice of $\Delta_{\mu_1,\mu_2}$ is the same for the pairs of coweights $(\eta_1,\eta_2)$ and $(\eta_1+\eta_1',\eta_2+\eta_2')$. By swapping the roles of $\eta$ and $\eta'$ in the above, the choice of $\Delta_{\mu_1,\mu_2}$ is also the same for the pairs $(\eta_1',\eta_2')$ and $(\eta_1+\eta_1',\eta_2+\eta_2')$.
\vskip1ex
Finally, we check that the diagram in the statement of the theorem commutes for any pair of antidominant coweights $\eta_1, \eta_2$. Let $\eta_1',\eta_2'$ be antidominant coweights such that $\mu_k+\eta_k+\eta_k'$ ($k=1,2$) are antidominant. Consider the diagram
\[
\xymatrix{
Y_\mu  \ar[d]_{\iota_{\mu,\eta_1,\eta_2}} \ar[rrr]^{\Delta_{\mu_1,\mu_2}} &&& Y_{\mu_1}\otimes Y_{\mu_2} \ar[d]^{(\iota_{\mu_1,\eta_1,0})\otimes (\iota_{\mu_2,0,\eta_2})}\\
Y_{\mu+ \eta} \ar[d]_{\iota_{\mu+\eta,\eta_1',\eta_2'}} \ar[rrr]_{\Delta_{\mu_1+\eta_1,\mu_2+\eta_2}} &&& Y_{\mu_1+\eta_1}\otimes Y_{\mu_2+\eta_2} \ar[d]^{(\iota_{\mu_1+\eta_1,\eta_1',0})\otimes (\iota_{\mu_2+\eta_2,0,\eta_2'})}
\\
Y_{\mu+\eta+\eta'} \ar[rrr]_{\Delta_{\mu_1+\eta_1+\eta'_1,\mu_2+\eta_2+\eta'_2}}& & & Y_{\mu_1+\eta_1+\eta_1'} \otimes Y_{\mu_2+\eta_2+\eta_2'}
}
\]
Since $\iota_{\mu+\eta, \eta_1',\eta_2'}\circ\iota_{\mu,\eta_1,\eta_2}=\iota_{\mu, \eta_1+\eta_1',\eta_2+\eta_2'}$ and $\iota_{\mu_1+\eta_1,\eta_1',0}\circ \iota_{\mu_1,\eta_1,0}=\iota_{\mu,\eta_1+\eta_1',0}$ and $\iota_{\mu_2+\eta_2,0,\eta_2'}\circ \iota_{\mu_2,0,\eta_2}=\iota_{\mu_2,0,\eta_2+\eta_2'}$, the outer square and the lower square are commutative. Since $\iota_{\mu_1+\eta_1,\eta_1',0}\otimes \iota_{\mu_2+\eta_2,0,\eta_2'}$ is injective, we see that the upper square is also commutative.
\end{proof}

\rem{BK coproduct}
Brundan and Kleshchev define a coproduct for shifted 
$\mathfrak{gl}_n$--Yangians in~\cite[Theorem 11.9]{BK}, which is analogous to 
our coproduct in the $\fsl_n$ case when $\mu = \mu_1 + \mu_2$ are all dominant. 
Namely, form the associated lower (resp.\ upper) triangular shift matrices 
$\sigma'$, $\sigma''$ by extending $s_{i+1, i}' = \mu_{1,i}$ and $s_{i, i+1}'' = \mu_{2, i}$, and take $\sigma = \sigma' + \sigma''$.  Then Brundan and Kleshchev's coproduct $Y_n(\sigma) \rightarrow Y_n(\sigma') \otimes Y_n(\sigma'')$ is defined by embedding into $Y(\fgl_n) \rightarrow Y(\fgl_n)\otimes Y(\fgl_n)$.  However, the standard inclusion of Hopf algebras $Y(\fsl_n)\hookrightarrow Y(\fgl_n)$ is {\em not} compatible with our shift map $\iota_{\mu, -\mu_1, -\mu_2}$.  
\erem

\begin{Prop} \label{coassociative2}
Suppose that $\mu=\mu_1+\mu_2+\mu_3$ where $\mu_2$ is antidominant. Then, the following diagram is commutative:
\[
\xymatrix{
Y_{\mu} \ar[rr]^{\Delta_{\mu_1,\mu_2+\mu_3}} \ar[d]_{\Delta_{\mu_1+\mu_2,\mu_3}}&& Y_{\mu_1}\otimes Y_{\mu_2+\mu_3} \ar[d]^{1\otimes \Delta_{\mu_2,\mu_3}}\\
Y_{\mu_1+\mu_2}\otimes Y_{\mu_3} \ar[rr]_{\Delta_{\mu_1,\mu_2}\otimes 1}&& Y_{\mu_1}\otimes Y_{\mu_2}\otimes Y_{\mu_3}
}
\]
\end{Prop}
\begin{proof}
Let $\eta_1,\eta_3$ be antidominant coweights such that $\mu_1'=\mu_1+\eta_1$ and $\mu_3'=\mu_3+\eta_3$ are antidominant. Consider the diagram
\[
\xymatrix{
& Y_{\mu_1'+\mu_2+\mu_3'} \ar[rr]^\Delta \ar[dd]_\Delta&  & Y_{\mu_1'}\otimes Y_{\mu_2+\mu_3'} \ar[dd]^{1\otimes \Delta}\\
Y_{\mu} \ar[rr]^{\qquad \Delta} \ar[dd]_{\Delta } \ar[ur]^{\iota_{\mu, \eta_1,\eta_3}}&& Y_{\mu_1}\otimes Y_{\mu_2+\mu_3} \ar[dd]^{1\otimes \Delta} \ar[ur]^{\iota_{\mu_1,\eta_1,0}\otimes \iota_{\mu_2+\mu_3,0,\eta_3}}&\\
& Y_{\mu_1'+\mu_2}\otimes Y_{\mu_3'} \ar[rr]^{\Delta\otimes 1 \qquad \qquad \qquad}& & Y_{\mu_1'}\otimes Y_{\mu_2}\otimes Y_{\mu_3'}\\
Y_{\mu_1+\mu_2}\otimes Y_{\mu_3} \ar[rr]_{\Delta\otimes 1} \ar[ur]_{\iota_{\mu_1+\mu_2,\eta_1,0}\otimes \iota_{\mu_3,0,\eta_3}}&& Y_{\mu_1}\otimes Y_{\mu_2}\otimes Y_{\mu_3} \ar[ur]_{\iota_{\mu_1,\eta_1,0}\otimes 1 \otimes \iota_{\mu_3,0,\eta_3}}&
}
\]

We have the commutativity of all faces of this cube except for that of
\[
\xymatrix{
Y_{\mu} \ar[rr]^{\Delta_{\mu_1,\mu_2+\mu_3}} \ar[d]_{\Delta_{\mu_1+\mu_2,\mu_3}}&& Y_{\mu_1}\otimes Y_{\mu_2+\mu_3} \ar[d]^{1\otimes \Delta_{\mu_2,\mu_3}}\\
Y_{\mu_1+\mu_2}\otimes Y_{\mu_3} \ar[rr]_{\Delta_{\mu_1,\mu_2}\otimes 1}&& Y_{\mu_1}\otimes Y_{\mu_2}\otimes Y_{\mu_3}
}
\]
Using the commutativity of the other faces and injectivity of shift maps, we see that the above square also commutes.
\end{proof}

\rem{notassociative}
In general, the coproducts are not coassociative.  More precisely, when $ \mu_2 $ is not antidominant, the diagram from Proposition \ref{coassociative2} does not commute.  This can already be seen for $\fg=\fsl_2,\ \mu_1=\mu_3=0,\ \mu_2=2$.
\erem

\sec{classical}{Classical limit}

In this section, we continue with $ \fg $ as simply-laced semisimple Lie algebra and we let $ G $ be the semisimple complex group of adjoint type whose Lie algebra is $ \fg $.

\ssec{generalities}{Generalities on filtrations}

Let $ A $ be a $\BC$-algebra and let $ F^\bullet A = \dots \subseteq F^{-1} A \subseteq F^0 A \subseteq F^1 A \subseteq \dots$ be a separated and exhaustive filtration, meaning that $ \cap_k F^k A = 0 $ and $\cup_k F^k A = A $.  We assume that this filtration is compatible with the algebra structure in the sense that $ F^k A \cdot F^l A \subset F^{k+l} A $ and $1 \in F^0 A$.

In this case, we define the {\em Rees algebra} of $ A$ to be the graded $ \BC[\hbar]$--algebra $ \Rees^F A := \oplus_k \hbar^k F^k A $, viewed as a subalgebra of $ A[\hbar,\hbar^{-1}]$. We also define the associated graded of $ A $ to be the graded algebra $ \gr^F A := \bigoplus F^k A / F^{k-1} A $.  Note that we have a canonical isomorphism of graded algebras $ \Rees^F A/ \hbar \Rees^F A \cong \gr^F A  $.

We say that the filtered algebra $ A $ is {\em almost commutative} if $ \gr^F A $ is commutative.

Now suppose that our algebra $ A $ is also graded, $ A = \oplus_n A_n $ and define $ F^k A_n := F^k A \cap A_n $.  Assume that for each $ k $, $ F^k A = \oplus_n F^k A_n $.  Define a new filtration $ G $ on $ A $ by setting $ G^k A = \oplus_{n+r = k} F^r A_n $.

\lem{change filtration}
With the above setup, we have canonical algebra isomorphisms $ \Rees^F A  \cong \Rees^G A $ and $ \gr^F A \cong \gr^G A $.
\elem

\prf
We prove the isomorphism for the associated graded (the Rees one is similar).  Define $B_{k,n} = F^k A_n / F^{k-1} A_n $.  Then $ \gr^F A = \oplus_{k,n} B_{k,n} $.  Now $ B_{k,n} = G^{k+n} A_n / G^{k+n-1} A_n $.  Thus we also see that $ \gr^G A = \oplus B_{k,n} $.  This gives us the isomorphism of vector spaces $ \gr^F A \rightarrow \gr^G A $ which is easily seen to be an algebra isomorphism as well.
\epr

\rem{freeness}
Suppose that the filtration $ F^\bullet A $ admits an expansion as a $\BC$-filtered vector space;  this means that we can find a filtered vector space isomorphism $ \gr^F A \rightarrow A $ (this condition is always satisfied if the filtration is bounded below).   If the filtration admits an expansion, then it is easy to see that $ \Rees^F A $ is a free $\BC[\hbar] $-module.

Moreover, suppose we have two filtered algebras $ F^\bullet A $ and $ F^\bullet B $.  We can define a filtration on $ A \otimes B $ by $ F^n (A \otimes B) = +_{k + l = n} F^k A \otimes F^l B $.  If the filtrations of $A$ and $B$ admit expansions, then we have $ \Rees (A \otimes B) \cong \Rees A \otimes \Rees B $ and $\gr (A \otimes B) \cong \gr A \otimes \gr B $.
\erem

\ssec{filtr}{Filtrations on the shifted Yangian}
Let $ \mu $ be any coweight.

Given any pair of coweights $ \nu_1, \nu_2 $ such that $ \nu_1 + \nu_2 = \mu $, we define a filtration $ F_{\nu_1, \nu_2} Y_\mu $ by defining degrees on the PBW variables as follows
\eq{grading on Ymu}
\deg E_\alpha^{(q)} = \langle \nu_1, \alpha \rangle + q, \ \deg F_\beta^{(q)} = \langle \nu_2, \beta \rangle + q, \ \deg H_i^{(p)} = \langle \mu, \alpha_i \rangle + p
\end{equation}
More precisely, we define $F^k_{\nu_1,\nu_2} Y_\mu$ to be the span of all {\em ordered} monomials in the PBW variables whose total degree is at most $k$.  A priori it is not clear that this filtration is independent of the choice of PBW variables, nor that it is independent of
the order used to form the monomials, nor that it is even an algebra filtration. We establish these properties in \refp{commutative} below.

Our goal now is to prove that $ \gr^{F_{\nu_1, \nu_2}}Y_\mu $ is commutative and to construct an isomorphism between this ring and the coordinate ring of a certain infinite type affine variety.

Now, assume that $ \mu $ is antidominant.  Define a filtration $F_\mu$ on $ Y\otimes \BC[H_i^{(0)}: i\in I]$ by taking a tensor product of the filtration $ F_{0,0} Y $ with the filtration on $\BC[H_i^{(0)}: i\in I]$ given by setting
$ \deg H_i^{(0)} =  \langle \mu, \alpha_i \rangle $.

Define a filtration $ F_\mu $ on $ \widetilde{Y} $ by defining degrees on the PBW variables as follows
$$\deg \widetilde{E}_\alpha^{(q)} = \langle \mu, \alpha \rangle + q,\  \deg \widetilde{F}_\beta^{(q)} = q, \ \deg \widetilde{H}_i^{(p)} = \langle \mu, \alpha_i \rangle + p   $$
As above, the filtered piece $F_\mu^k \widetilde{Y}$ is defined to be the span of those ordered monomials in the PBW variables whose total degree is at most $k$.

\lem{strict}
The inclusion $ \widetilde{Y} \hookrightarrow Y\otimes_\BC \BC[H_i^{(0)}: i\in I] $ is compatible with the filtrations $ F_\mu $ on both algebras.  Moreover, this inclusion is strict, i.e. for each $k $, $$ F_\mu^k(\widetilde{Y}) = F_\mu^k(Y\otimes \BC[H_i^{(0)}: i\in I]) \cap \widetilde{Y}.$$ Thus the resulting map $$ \gr^{F_\mu} \widetilde{Y} \rightarrow \gr^{F_\mu}(Y\otimes \BC[H_i^{(0)}: i\in I]) $$ is injective.
\elem
\prf
Both filtrations are defined by the degrees of monomials, and therefore it suffices to verify that the degree of a monomial from $\widetilde{Y}$ is equal to the degree of its image in $Y\otimes \BC[H_i^{(0)}]$.
\epr

\cor{tilde commutative}
The filtration $F_\mu \widetilde{Y}$ is an algebra filtration, and  $ \widetilde{Y}$ is almost commutative.  Moreover, $F_\mu \widetilde{Y}$ is independent of the choice of PBW variables and is also independent of the order used to form the monomials.
\ecor
\prf
Since $\widetilde{Y}\hookrightarrow Y\otimes \BC[H_i^{(0)}]$ is an inclusion of algebras, it follows immediately from \refl{strict} that $F_\mu \widetilde{Y}$ is an algebra filtration.  We know from \cite{kwwy} that $ \gr^{F_{0,0}} Y $ is commutative.  (In fact, it is isomorphic to $ \BC[G_1[[z^{-1}]]]$.)  Thus $ \gr^{F_\mu}(Y\otimes \BC[H_i^{(0)}: i\in I]) $ is commutative, so \refl{strict} implies that $\gr^{F_\mu} \widetilde{Y}$ is commutative.  Finally, independence of choice of PBW monomials also follows for the corresponding property for $F_{0,0} Y$.
\epr

Now, we show that $ Y_\mu $ is almost commutative.
\prop{commutative}
The filtration $F_{\mu,0} Y_\mu$ is an algebra filtration, and  $ Y_\mu$ is almost commutative.  Moreover, $F_{\mu,0} Y_\mu$ is independent of the choice of PBW variables and is also independent of the order used to form the monomials.
\eprop

\prf
First consider the case of $\mu$ antidominant.  We then have a surjective map of algebras $\widetilde{Y}\twoheadrightarrow Y_\mu$, under which $F_{\mu,0} Y_\mu$ is the quotient filtration of $F_\mu \widetilde{Y}$.  All three properties follow from \refc{tilde commutative}.

Next, consider the case of general $\mu$.  Choose $\mu_1$ antidominant such that $\mu+\mu_1 $ is antidominant, and consider the shift homomorphism $\iota_{\mu, \mu_1,0}: Y_\mu \rightarrow Y_{\mu+\mu_1}$. This map is injective by \refc{cor: injectivity of shift maps}, and it is compatible with the filtrations $F_{\mu,0} Y_\mu \rightarrow F_{\mu+\mu_1,0} Y_{\mu+\mu_1}$.  Moreover it is {\em strict}, by the same reasoning as \refl{strict}. We now reason as in the proof of \refc{tilde commutative}, proving the claim.
\epr

Now, let $ \nu_2 $ be any coweight and let $ \nu_1 = \mu - \nu_2 $.  Define a grading on $ Y_\mu $ by setting the graded degree of the generators as follows
\eq{grading by coweight}
\deg E_i^{(q)} = \langle -\nu_2, \alpha_i \rangle,\ \deg F_i^{(q)} = \langle \nu_2, \alpha_i \rangle,\ \deg H_i^{(p)} =0
\end{equation}
This is easily seen to be a grading on $ Y_\mu $.  (This grading is the eigenspaces of the adjoint action of the element $ \sum_i \langle -\nu_2, \omega_i \rangle H_i^{(-\langle \mu, \alpha_i \rangle + 1)}  $ where $\omega_i$ is a fundamental
weight).

The filtration $F_{\nu_1, \nu_2} Y_\mu $ comes from the filtration $ F_{\mu, 0} Y_\mu $ and the above grading using the construction given in \refss{generalities}.  Thus, we get a canonical isomorphism $ \gr^{F_{\nu_1, \nu_2}} Y_\mu \cong \gr^{F_{\mu, 0}} Y_\mu $ by \refl{change filtration} and in particular the former is commutative.  Since all $ \gr^{F_{\nu_1, \nu_2}} Y_\mu $ are canonically isomorphic (as algebras, but not as graded algebras), we will write $ \gr Y_\mu $ to denote any one of them, when we are not concerned with the grading.  Similarly, all Rees algebras $ \Rees^{F_{\nu_1, \nu_2}} Y_\mu $ are all canonically isomorphic and we will write $ \bY_\mu := \Rees Y_\mu $.

\cor{polynomial ring}
The algebra $\gr Y_\mu$ is a polynomial ring in the PBW variables.
\ecor

\ssec{Wmu}{The variety $ \CW_\mu $}
For any algebraic group $H $, we write $ H_1[[z^{-1}]] $ for the kernel of the evaluation map $ H[[z^{-1}]] \rightarrow H $.

We define the (infinite type) scheme
\eq{def of Wmu}
\CW_\mu := U_1[[z^{-1}]] T_1[[z^{-1}]] z^\mu U_{-,1}[[z^{-1}]]  \subset G((z^{-1}))
\end{equation}
We will also need a different description of this scheme.  The inclusion $ U_1[[z^{-1}]] \rightarrow U((z^{-1})) $ gives rise to an isomorphism $ U_1[[z^{-1}]] \cong U((z^{-1})) / U[z] $.  Thus we can identify $ \CW_\mu $ with the quotient $ U[z] \lmod U((z^{-1})) T_1[[z^{-1}]] z^\mu U_-((z^{-1})) / U_-[z] $ and we write $ \pi $ for this isomorphism.

The scheme $ \CW_\mu $ is the moduli space of the following data
(cf.~\cite[2(ii)]{bfn16}):
(a) a $G$-bundle $\CP$ on $\BP^1$; (b) a trivialization
$\sigma\colon \CP_{\on{triv}}|_{\widehat\BP{}^1_\infty}\iso\CP|_{\widehat\BP{}^1_\infty}$
in the formal neighbourhood of $\infty\in\BP^1$;
(c) a $B$-structure $\phi$ on $\CP$ of degree $w_0\mu$ having fiber
$B_-\subset G$ at $\infty\in\BP^1$ (with respect to the trivialization $\sigma$
of $\CP$ at $\infty\in\BP^1$). This is explained in~\cite[2(xi)]{bfn16}.
In particular, $\CW_\mu$ contains the finite dimensional affine
varieties $\ol\CW{}_\mu^\lambda$ (generalized slices) for the dominant coweights
$\lambda$, and the closed subvariety $\ol\CW{}_\mu^\lambda\subset\CW_\mu$ is cut
out by the condition that $\sigma$ extends as a rational trivialization with
a unique pole at $0\in\BP^1$, and the order of the pole of $\sigma$ at
$0\in\BP^1$ is $\leq\lambda$.

For $ \mu_1, \mu_2 $ antidominant, we define shift maps
$ \iota_{\mu, \mu_1, \mu_2}\colon \CW_{\mu + \mu_1 + \mu_2} \rightarrow \CW_\mu $ by $ g \mapsto \pi(z^{-\mu_1} g z^{-\mu_2}) $.

\rem{}
Note that $ \CW_0 $ is the group $ G_1[[z^{-1}]] $.  Moreover, for $ \mu $ dominant, we can identify $ \CW_\mu $ with the $G_1[[z^{-1}]] $ orbit of $ z^\mu $ in the thick affine Grassmannian $ G((z^{-1}))/G[z] $.  In this case, the shift map $ \iota_{\mu, 0,-\mu}\colon \CW_0 \rightarrow \CW_\mu $ is exactly the action map.  In fact, $ \CW_\mu = \mathcal{G}_\mu z^\mu $, where $ \mathcal{G}_\mu $ is the subgroup of $ G_1[[z^{-1}]] $ defined in \cite[B(viii)(a)]{bfn16}.
\erem

Recall the multiplication morphisms $m_{\mu_1,\mu_2}^{\lambda_1,\lambda_2}\colon
\ol\CW{}_{\mu_1}^{\lambda_1}\times\ol\CW{}_{\mu_2}^{\lambda_2}\to
\ol\CW{}_{\mu_1+\mu_2}^{\lambda_1+\lambda_2}$ constructed in~\cite[2(vi)]{bfn16}.
We define the multiplication morphism
$m_{\mu_1, \mu_2}\colon \CW_{\mu_1} \times \CW_{\mu_2} \rightarrow \CW_{\mu_1 + \mu_2}$
by the formula $m_{\mu_1,\mu_2}(g_1, g_2) = \pi(g_1 g_2)$.
Comparing the constructions of~\cite[2(vi) and 2(xi)]{bfn16}, we see that
$m_{\mu_1,\mu_2}^{\lambda_1,\lambda_2}$ is the restriction of $m_{\mu_1,\mu_2}$.

\lem{compatible}
The shift maps and multiplication maps are compatible.  More precisely, let $ \mu_1, \mu_2 $ be any coweights and let $ \nu_1, \nu_2 $ be antidominant coweights. The following diagram commutes.
\[
\xymatrix{
\CW_{\mu_1 + \nu_1} \times \CW_{\mu_2 + \nu_2} \ar[r] \ar[d]_{\iota_{\mu_1, \nu_1, 0} \times \iota_{\mu_2, 0, \nu_2}} & \CW_{\mu_1 + \mu_2 + \nu_1 + \nu_2} \ar[d]^{\iota_{\mu_1 + \mu_2, \nu_1, \nu_2}} \\  \CW_{\mu_1} \times \CW_{\mu_2} \ar[r] & \CW_{\mu_1 + \mu_2}
}
\]

\elem
\prf
A simple diagram chase show that for $ (g_1,g_2) \in \CW_{\mu_1 + \nu_1} \times \CW_{\mu_2 + \nu_2}$, both paths are computed by $ \pi(z^{-\nu_1} g_1 g_2 z^{-\nu_2})$.  (Here we use that if $ u \in U[z] $, then $ z^{-\nu_1} u z^{\nu_1} \in U[z] $.)
\epr

For $ s \in \Cx$, and $ g(z) \in G((z^{-1})) $, we define $ \kappa_s(g(z)) = g(sz) $.  This loop rotation action does not preserve $ \CW_\mu \subset G((z^{-1}))$.  But given a pair of coweights $ \nu_1, \nu_2 $ such that $ \nu_1 + \nu_2 = \mu $, we define an action $\kappa^{\nu_1, \nu_2} $ of $ \Cx $ on $ \CW_\mu $ by
$$
\kappa^{\nu_1, \nu_2}_s(g) = s^{-\nu_1} \kappa_s(g) s^{-\nu_2}
$$

\ssec{sclass}{Classical limit}
Let $p_i\colon U\to\BC$, $p_i^-\colon U_-\to\BC$, $\sfp_i\colon T\to\BC^\times$ be the projections according to
a simple root $\alpha_i$.  Then we get maps $ p_i^{(r)}\colon U_1[[z^{-1}]] \rightarrow \BC $ given by taking the coefficient of $ z^{-r} $ in $ p_i$.  Similarly, we get functions $ {p_i^-}^{(r)} $ and $ \sfp_i^{(r)} $.  Using the Gauss decomposition of an element $ u h z^\mu u_- \in \CW_\mu $, we get functions $ p_i^{(r)}, {p_i^-}^{(r)} $ and $ \sfp_i^{(r)} $ on $ \CW_\mu$  by
\eq{functions via gauss}
p_i^{(r)}(g) := p_i^{(r)}( u), \ \ {p_i^-}^{(r)}(g) := {p_i^-}^{(r)}(u_-), \ \ \sfp_i^{(r)}(g) := \sfp_i^{(r)}( h z^\mu)
\end{equation}
These functions can also be described using generalized minors (i.e. matrix coefficients), analogously to \cite{kwwy}.

As described in \cite{kwwy}, $\CW_0 = G_1[[z^{-1}]]$ can be given the structure of a Poisson-Lie group, corresponding to Yang's Manin triple. The ring of
functions $\BC[\CW_0]$ is graded via the loop rotation action.
The following result is a reformulation of \cite[Theorem 3.9]{kwwy}:
\th{kwwy-thm}
There is an isomorphism of graded Poisson-Hopf algebras $\gr^{F_{0,0}} Y \cong \BC[\CW_0]$, such that
\eq{eq: kwwy thm}
H_i^{(r)} \mapsto \sfp_i^{(r)}, \ \ E_i^{(r)} \mapsto p_i^{(r)}, \ \ F_i^{(r)} \mapsto {p_i^-}^{(r)}
\end{equation}
\eth


\prop{gauss and triangular decomp}
\mbox{}

\begin{itemize}
\item[(a)] For any coweight $\nu$, there is an isomorphism of graded Poisson algebras $\gr^{F_{\nu, -\nu}} Y \cong  \BC[\CW_0]$ such that \refe{eq: kwwy thm} holds, and where the grading on $\BC[\CW_0]$ comes from the $\kappa^{\nu, -\nu}$ action.

\item[(b)] This restricts to a graded isomorphism
\[
\xymatrix{
\gr^{F_{\nu,-\nu}} Y^> \ar@{^{(}->}[d] &\ar[r]^{\sim}  && \BC[U_1[[z^{-1}]]] \ar@{^{(}->}[d] \\  \gr^{F_{\nu,-\nu}} Y& \ar[r]^{\sim} && \BC[\CW_0]
}
\]
where the right-hand vertical arrow corresponds to the map $\CW_0 \rightarrow U_1[[t^{-1}]], \quad u h u_- \mapsto u$.
 Similarly, there are isomorphisms
$$
\gr^{F_{\nu,-\nu}} \Cartan \cong \BC[T_1[[z^{-1}]]], \ \  \gr^{F_{\nu, -\nu}} Y^< \cong \BC[ U_{-,1}[[z^{-1}]]] $$
\end{itemize}
\eprop
Part (b) reflects the triangular decomposition $ Y $ on the one hand, and the Gauss decomposition of $\CW_0$ on the other.
\prf
For (a), it suffices to show that the grading \refe{grading by coweight} corresponds to the $\BC^\times$--action on $\CW_0$ given by $s\cdot g = s^{-\nu} g s^{\nu}$. First, we claim that the latter is a Poisson action: this follows from the explicit formula \cite[Proposition 2.13]{kwwy} for the Poisson bracket in terms of generalized minors.  Now, under the corresponding grading on $\BC[\CW_0]$ the degrees of generators $p_i^{(r)}, {p_i^-}^{(r)}, \sfp_i^{(r)}$ agree with the grading \refe{grading by coweight}.  Since both are Poisson gradings, and these are Poisson generators of $\BC[\CW_0]$, the two gradings agree.

We now prove (b), in the case of $Y^>$.  Recall that $G_1[[z^{-1}]]$ is a Poisson algebraic group, so $z^{-1}\fg[[z^{-1}]]$ is a Lie bialgebra~(see \cite[2C]{kwwy}).  Under its cobracket, we have
$$ \delta( z^{-1} \mathfrak{b}^-[[z^{-1}]] ) \subset (z^{-1} \mathfrak{b}^-[[z^{-1}]]) \otimes (z^{-1} \fg[[z^{-1}]]) + (z^{-1} \fg[[z^{-1}]]) \otimes (z^{-1} \mathfrak{b}^-[[z^{-1}]])$$
By \cite[Theorem 6]{sts}, this implies that there is an induced structure of Poisson homogeneous space on $G_1[[z^{-1}]] / B_1^-[[z^{-1}]] $. In other words, $\BC[ G_1[[z^{-1}]] ]^{B_1^-[[z^{-1}]]}$ is a Poisson subalgebra of $\BC[G_1[[z^{-1}]]$.

The map $G_1[[z^{-1}]] \rightarrow U_1[[z^{-1}]]$, $u h u_- \mapsto u$ identifies $\BC[U_1[[z^{-1}]]] \cong \BC[ G_1[[z^{-1}]] ]^{B_1^-[[z^{-1}]]}$.
 Since the functions $p_i^{(r)}$ lie in this subalgebra, $\BC[U_1[[z^{-1}]]]$
contains the Poisson subalgebra that they generate.  Therefore we can identify
$\gr^{F_{\nu,-\nu}} Y^> \subset \BC[U_1[[z^{-1}]]]$. We will prove equality by a
dimension count.  It suffices to consider the filtration $F^{0,0} Y^>$ and the
loop rotation action $\kappa^{0,0}$. By the PBW theorem, $\gr^{F_{0,0}} Y^>$ has
Hilbert series
\eq{hilbert series}
\prod_{i=1}^\infty \frac{1}{(1-q^i)^{\dim \mathfrak{n}}}
\end{equation}
Since $U_1[[z^{-1}]]$ is a pro-unipotent group, the Hilbert series of
$\BC[U_1[[z^{-1}]]]$ for the loop rotation is the same as that of
$\operatorname{Sym} (z^{-1} \mathfrak{n}[[z^{-1}]])$. This is also given
by~\refe{hilbert series}, proving the claim.
\epr

Consider a coweight $\mu$. By \refr{maps from Yangian 2}, there is a surjection of algebras $Y^> \twoheadrightarrow Y_\mu^>$ defined by $E_i^{(p)} \mapsto E_i^{(p)}$.  By the PBW theorem for $Y_\mu$, it follows that this map is an isomorphism.  Moreover, for any coweights $\nu_1, \nu_2$ such that $\nu_1 + \nu_2 = \mu$, we see from \refe{grading on Ymu} that it is an isomorphism of filtered algebras $F^{\nu_1, -\nu_1} Y^> \stackrel{\sim}{\longrightarrow} F^{\nu_1, \nu_2} Y_\mu^>$, where these filtrations are inherited as subspaces of $Y$ and $Y_\mu$, respectively.

By the Gauss decomposition there is a projection map $\CW_\mu \twoheadrightarrow U_1[[z^{-1}]]$, $u h z^\mu u_- \mapsto u$.  This provides an embedding $\BC[ U_1[[z^{-1}]]] \hookrightarrow \BC[\CW_\mu]$.  Consider the composition of maps:
\eq{class 1}
\gr^{F_{\nu_1,\nu_2}} Y_\mu^> \stackrel{\sim}{\rightarrow} \gr^{F_{\nu_1, -\nu_1}} Y^> \stackrel{\sim}{\rightarrow} \BC[U_1[[z^{-1}]]] \hookrightarrow \BC[ \CW_\mu]
\end{equation}
where the second map comes from \refp{gauss and triangular decomp}.  Note that under this composition, $E_i^{(r)} \mapsto p_i^{(r)}$.  This map is graded, where $\BC[\CW_\mu]$ is graded by the action $\kappa^{\nu_1,\nu_2}$.

Analogously, there are compositions
\begin{align}
\gr^{F_{\nu_1, \nu_2}} Y_\mu^< & \stackrel{\sim}{\rightarrow} \gr^{F_{-\nu_2, \nu_2}} Y^< \stackrel{\sim}{\rightarrow} \BC[U_{-,1}[[z^{-1}]]] \hookrightarrow \BC[\CW_\mu], \label{class 2} \\
\gr^{F_{\nu_1, \nu_2}} \Cartan_\mu & \stackrel{\sim}{\rightarrow} \gr^{F_{0,0}} \Cartan \stackrel{\sim}{\rightarrow} \BC[T_1[[z^{-1}]]] \hookrightarrow \BC[\CW_\mu] \label{class 3}
\end{align}
which take $F_i^{(r)} \mapsto {p_i^-}^{(r)}$ and $H_i^{(r)} \mapsto \sfp_i^{(r)}$.

From \refc{polynomial ring}, there is a triangular decomposition (of algebras)
$$ \gr^{F_{\nu_1, \nu_2}} Y_\mu \cong (\gr^{F_{\nu_1,\nu_2}} Y_\mu^> )\otimes (\gr^{F_{\nu_1, \nu_2}} \Cartan_\mu) \otimes (\gr^{F_{\nu_1, \nu_2}} Y_\mu^<) $$
By the Gauss decomposition $\CW_\mu$, we get:
\th{class}
For any coweights $ \nu_1, \nu_2 $ such that $\nu_1 + \nu_2 = \mu $, the tensor product of the maps \refe{class 1}, (\ref{class 2}) and (\ref{class 3}) yields an isomorphism of graded algebras
\eq{eq: class}
\gr^{F_{\nu_1, \nu_2}} Y_\mu \stackrel{\sim}{\longrightarrow} \BC[\CW_\mu]
\end{equation}
Here the grading on $\BC[\CW_\mu]$ comes from the $\kappa^{\nu_1,\nu_2}$ action.  Moreover, the isomorphism is compatible with the shift maps $ \iota_{\mu, \mu_1, \mu_2} $ on both sides.
\eth

\prf
The only thing left to prove is the compatibility of the shift maps $ \iota_{\mu, \mu_1, \mu_2}\colon Y_\mu \rightarrow Y_{\mu + \mu_1 + \mu_2} $ and $ \iota_{\mu, \mu_1, \mu_2}\colon \CW_{\mu + \mu_1 + \mu_2} \rightarrow \CW_\mu $ with the above isomorphism $\gr Y_\mu \cong \BC[\CW_\mu]$.  Since this isomorphism is constructed using the Gauss decomposition, it suffices to prove the compatibility on each piece separately.
For $Y_\mu^=$ and $T_1[[z^{-1}]]$ it follows from the construction
of~(\ref{class 3}).
Now it suffices to check the compatibility with the isomorphisms~\refe{class 1}
and~(\ref{class 2}).  Since these are similar, we will just concentrate on the
isomorphism \eqref{class 2}.

For any $ \eta $ antidominant, define a map $ \psi_\eta\colon U_{-,1}[[z^{-1}]] \rightarrow U_{-,1}[[z^{-1}]] $ by $ \psi_\eta(u) = \pi(z^\eta u z^{-\eta}) $, where as usual $ \pi $ denotes the projection $ \pi\colon U_-((z^{-1})) \rightarrow U_-((z^{-1}))/U_-[z] \cong U_{-,1}[[z^{-1}]] $.

For any coweight $ \mu $ and any antidominant $\mu_1, \mu_2 $ we have the commutativity of the diagram
\[
\xymatrix{
\CW_{\mu + \mu_1 + \mu_2} \ar[d]^{\iota_{\mu, \mu_1, \mu_2}} \ar[r] & U_{-,1}[[z^{-1}]] \ar[d]^{\psi_{\mu_2}} \\
\CW_\mu \ar[r] & U_{-,1}[[z^{-1}]]
}
\]

On the other hand, consider the shift map $ \psi_\eta\colon Y^< \rightarrow Y^< $ given by $ F_i^{(q)} \mapsto F_i^{(q - \langle \eta, \alpha_i \rangle)} $.  Then $ \psi_{\mu_2} $ is the restriction of the shift map $ \iota_{\mu, \mu_1, \mu_2} $ to $ Y^< \cong Y_\mu^<$.

 Thus, in order to show that \eqref{class 2} is compatible with the shift $ \iota_{\mu, \mu_1, \mu_2} $, is suffices to show that the isomorphism $ \gr Y^< \cong \BC[U_{-,1}[[z^{-1}]]] $ is compatible with the two $ \psi_\eta$ maps, for any antidominant $ \eta $.

Since $ -\eta $ is dominant, it follows from~\cite[Theorem~3.12]{kwwy}, that we have a commutative diagram
\[
 \xymatrix{
\gr Y_{-\eta} \ar[r] \ar[d]^{\iota_{-\eta, 0, \eta}} & \BC[\CW_{-\eta}] \ar[d] \\
\gr Y \ar[r] & \BC[G_1[[z^{-1}]]
}
\]
where the right vertical arrow is just (dual to) the action map $ g \mapsto gz^{-\eta} $. This action map is compatible with the shift map $ \psi_\eta $ on $ U_{-,1}[[z^{-1}]] $.  Thus we deduce that the isomorphism $ \gr Y^< \cong \BC[U_{-,1}[[z^{-1}]]] $ is compatible with the two shift maps $ \psi_\eta $ and this completes the proof.
\epr

\rem{PoissonWmu}
The above theorem provides $ \CW_\mu $ with a Poisson structure.  It is compatible with the Poisson structure on $ \oW^\lambda_\mu $ constructed in \cite{bfn16}.  This is because the Poisson structure on $ \oW^\lambda_\mu $ comes from its quantization (the quantized Coulomb branch) and we have a surjective map from $ \Rees Y_\mu $ to this quantized Coloumb branch provided by \cite[Theorem~B.18]{bfn16}.
\erem

\lem{lem-sy-gen}
Let $\mu$ be an antidominant coweight. Then the classical shifted Yangian $ \gr^F Y_\mu \cong \BC[\CW_{\mu}]$ is generated by $E_i^{(1)} = p_i^{(1)}$, $F_i^{(1)} = {p_i^-}^{(1)}$, $H_{i}^{(-\la\mu,\alpha_i\ra+1)} = \sfp_i^{(1)}$ and $H_{i}^{(-\la\mu,\alpha_i\ra+2)} = \sfp_i^{(2)}$ as a Poisson algebra.
\elem
\prf
From the PBW theorem, $ \gr^F Y_\mu $ is generated by the PBW variables.  These variables are all constructed from the generators $ E_i^{(r)}, F_i^{(r)}, H_i^{(r)} $ using Poisson brackets.  So it suffices to show that we can construct these
generators.

Indeed, $\{H_{i}^{(-\la\mu,\alpha_i\ra+2)},E_i^{(r)}\}$ is $2E_i^{(r+1)}$ plus some expression in $H_{i}^{(-\la\mu,\alpha_i\ra+1)}$ and $E_i^{(r)}$, and the same
holds for $F_i^{(r)}$, hence the algebra generated by the above elements contains $E_i^{(r)}$, $F_i^{(r)}$ for all $r\in\BZ_{>0}$. Every $H_i^{(s)}$ with positive $s$ is a bracket of some $E$ and $F$, hence we have $H_{i}^{(-\la\mu,\alpha_i\ra+r)}$ for all $r\in\BZ_{>0}$.
\epr

\ssec{comult}{Classical multiplication and the coproduct}
Let $ \mu_1, \mu_2 $ be coweights.

The multiplication map $ m\colon \CW_{\mu_1} \times \CW_{\mu_2} \rightarrow \CW_{\mu_1 + \mu_2} $ gives us an algebra map
$$
\Delta^1_{\mu_1, \mu_2}\colon \BC[\CW_{\mu_1 + \mu_2}] \rightarrow \BC[\CW_{\mu_1}] \otimes \BC[\CW_{\mu_2}]
$$

On the other hand, the coproduct
$$
Y_{\mu_1 + \mu_2} \rightarrow Y_{\mu_1} \otimes Y_{\mu_2}
$$
is compatible with the filtrations $ F_{\mu_1, \mu_2} Y_{\mu_1 + \mu_2} , F_{\mu_1, 0} Y_{\mu_1}, F_{0,\mu_2} Y_{\mu_2} $ and thus gives rise to a map
$$
\Delta^2_{\mu_1, \mu_2}\colon \gr Y_{\mu_1 + \mu_2} \rightarrow \gr Y_{\mu_1} \otimes \gr Y_{\mu_2}
$$

Under the isomorphism \reft{class}, this gives us another map $ \BC[\CW_{\mu_1 + \mu_2}] \rightarrow \BC[\CW_{\mu_1}] \otimes \BC[\CW_{\mu_2}] $.

When $ \mu_1 = \mu_2 = 0 $, we have $ \CW_{\mu_1} = \CW_{\mu_2} = \CW_{\mu_1 + \mu_2} = G_1[[z^{-1}]] $ and the multiplication map is just the ordinary multiplication map on $ G_1[[z^{-1}]] $.  On the other hand, the Drinfeld-Gavarini duality (also called Quantum Duality Principle; see \cite{kwwy}) shows us that the coproduct on $ \gr Y $ is just the usual coproduct on $\BC[G_1[[z^{-1}]]]$.  Thus we conclude that $ \Delta^1_{0,0} = \Delta^2_{0,0}$.  So it is natural to expect that for all $ \mu_1, \mu_2$, we have $ \Delta^1_{\mu_1, \mu_2} = \Delta^2_{\mu_1, \mu_2}$.  (In \refc{quantized multiplication}, we will show that this holds when $ \fg = \fsl_2 $ and $ \mu_1, \mu_2 $ are antidominant.)

\prop{agreegenerators}
If $\mu_1, \mu_2 $ are antidominant, then the $\Delta^1_{\mu_1, \mu_2} $ and $ \Delta^2_{\mu_1, \mu_2} $ agree on the Poisson generators $ p_i^{(1)}, {p_i^-}^{(1)},\sfp_i^{(1)}, \sfp_i^{(2)}$.
\eprop

\prf
Let $\mu_1$ and $\mu_2$ be both antidominant. Then the multiplication map $\CW_{\mu_1}\times\CW_{\mu_2}\to\CW_{\mu_1+\mu_2}$ is given just by multiplication in $G((z^{-1}))$. Let $u^k h^k z^{\mu_k}u_-^k$, $k=1,2$ be any elements of $\CW_{\mu_k}$. Then the product is
$$
u^1z^{\mu_1}h^1 u_-^1u^2h^2z^{\mu_2}u_-^2.
$$
We take the Gaussian decomposition of the middle part of this expression,
i.e.\ we write
$$
h^1 u_-^1u^2h^2=u'h'u'_-,
$$
where $u'\in U_1[[z^{-1}]]$, $u'_-\in U_{-,1}[[z^{-1}]]$, $h'\in T_1[[z^{-1}]]$. Then the product is
$$
u^1\big(z^{\mu_1}u'z^{-\mu_1}\big) h'z^{\mu_1+\mu_2}\big(z^{-\mu_2}u'_-z^{\mu_2}\big)u_-^2.
$$

The first and second Fourier coefficients of $u', u'_-$ and $h'$ are easy to
compute. To write the answer we need to define the functions $p^{(r)}_\gamma$,
$p^{-(r)}_\gamma$ for any positive root $\gamma$. Fix an isomorphism
$c\colon \fn\to U$ between the formal neighborhood of $0\in\fn$  and the formal
neighborhood of $e\in U$ such that $d_0c=\on{Id}\colon \fn\to\fn$ (e.g.
$c=\exp\colon \fn\to U$). Let $P_\gamma\colon \fn\to\fn_\gamma=\BC$ be the
projection to the corresponding root space. Define $p^{(r)}_\gamma$ as the
coefficient of $z^{-r}$ in the composite map
$P_\gamma\circ c^{-1}\colon U_1[[z^{-1}]]\to\BC$, and similarly for
$p^{-(r)}_\gamma$. Note that $p^{(1)}_\gamma$ and $p^{-(1)}_\gamma$ do not depend on
the choice of $c$. We get the following formulas
\eq{31}
\Delta^1(\sfp_{i}^{(1)})=\sfp_{i}^{(1)}\otimes1+1\otimes \sfp_{i}^{(1)},
\end{equation}

\begin{equation}
\Delta^1(\sfp_{i}^{(2)})=\sfp_{i}^{(2)}\otimes1+1\otimes \sfp_{i}^{(2)}
+\sfp_{i}^{(1)}\otimes \sfp_{i}^{(1)}-\sum\limits_{\gamma>0}\la\alpha_i,\gamma\ra {p^{-}_\gamma}^{(1)}\otimes p_\gamma^{(1)},
\end{equation}

\begin{eqnarray}
\Delta^1(p_i^{(1)})=p_i^{(1)}\otimes1+\delta_{\la\mu_1,\alpha_i\ra,0}1\otimes p_i^{(1)},\\
\Delta^1({p^-_i}^{(1)})=\delta_{\la\mu_2,\alpha_i\ra,0}{p^-_i}^{(1)}\otimes1+1\otimes {p^-_i}^{(1)}.
\end{eqnarray}

Note that $ p_\gamma^{(1)} $ is the image of $ E_\gamma^{(1)} $ under the isomorphism from \reft{class} (and similarly ${p^{-}_\gamma}^{(1)}$).

Comparing with \refl{4generatorscoproduct} gives the desired result.
\epr

%
%
%

\conj{Poisson} The multiplication map $\CW_{\mu_1} \times \CW_{\mu_2} \rightarrow \CW_{\mu_1 + \mu_2} $ is Poisson.
\econj
We know that it is true when $ \mu_1 = \mu_2 = 0 $, since in this case it is just the usual multiplication in the Poisson group $ G_1[[z^{-1}]] $.

\prop{Deltaagree}
If \refco{Poisson} holds, then the two maps $ \Delta^1 $ and $ \Delta^2 $ agree.
\eprop

\prf
If $ \mu_1, \mu_2 $ are antidominant, then \refp{agreegenerators} shows that $ \Delta^1_{\mu_1, \mu_2} $ and $ \Delta^2_{\mu_1, \mu_2} $ agree on the Poisson generators for the algebra $ \BC[\CW_{\mu_1 + \mu_2}]$.  Since both maps are Poisson, they must agree.

Now, suppose that $ \mu_1, \mu_2 $ are arbitrary.  As in the proof of \reft{generalcoproduct}, we can embed $ \BC[\CW_{\mu_1 + \mu_2}] $ into an antidominant situation.  Both $\Delta^1 $ and $ \Delta^2 $ are compatible with this embedding.  For $ \Delta^1$, this follows from \refl{compatible}, while for $\Delta^2$ this follows from the construction in \reft{generalcoproduct}.  Thus the result follows.
\epr

\sec{toda}{Toda and comultiplication}

Throughout this section we work with shifted Yangians of $ \fsl_2 $ and the Toda lattice for $GL(n)$.

\ssec{def-sl2}{A presentation of $\fsl_2$ shifted Yangians}

Following~\cite[Definition~2.24]{m03}, we can write down the defining relations
of the shifted Yangian $ \bY_m(\fsl_2)$ of $\fsl_2$ in current form. In this
case $\mu = m\in\BZ$, and from now on we assume $m\leq0$, i.e.\ our Yangian
is antidominantly shifted.  We introduce the series
$E(u):=\sum\limits_{p=1}^\infty E^{(p)}u^{-p}$, $F(u):=\sum\limits_{p=1}^\infty F^{(p)}u^{-p}$, $H(u):=u^m+\sum\limits_{p=-m+1}^\infty H^{(p)}u^{-p}$. Then the defining relations can be written in the following form:
\eq{311}
[H(u),H(v)]=0,
\end{equation}
\eq{312}
[E(u),F(v)]=-\hbar\frac{H(u)-H(v)}{u-v},
\end{equation}
\eq{313}
[E(u),E(v)]=-\hbar\frac{(E(u)-E(v))^2}{u-v},
\end{equation}
\eq{314}
[F(u),F(v)]=\hbar\frac{(F(u)-F(v))^2}{u-v},
\end{equation}
\eq{315}
[H(u),E(v)]=-\hbar\frac{H(u)(E(u)-E(v))+(E(u)-E(v))H(u)}{u-v},
\end{equation}
\eq{316}
[H(u),F(v)]=\hbar\frac{H(u)(F(u)-F(v))+(F(u)-F(v))H(u)}{u-v}.
\end{equation}


\ssec{u-eps}{Some automorphisms of $\bY_m(\fsl_2)$} It is clear from the formulas above that the additive shifts of the variable $u$ act on the shifted Yangian $\bY_m$ by automorphisms. We denote the corresponding automorphisms by $T_{\varepsilon}: E(u)\mapsto E(u-\varepsilon), F(u)\mapsto F(u-\varepsilon), H(u)\mapsto H(u-\varepsilon)$.

\ssec{sl2-coprod}{Coproduct on $\bY_m(\fsl_2)$ for $m\leq 0$} The following formulas define the coproduct $\Delta$ on the usual Yangian $\bY(\fsl_2)=\bY_0(\fsl_2)$ (see \cite[Definition~2.24]{m03}).
\eq{delta-e} \Delta\colon E(u)\mapsto E(u)\otimes1+\sum\limits_{j=0}^\infty(-1)^jF(u+\hbar)^jH(u)\otimes E(u)^{j+1};
\end{equation}
\eq{delta-f} \Delta\colon F(u)\mapsto 1\otimes F(u)+\sum\limits_{j=0}^\infty(-1)^j F(u)^{j+1}\otimes H(u)E(u+\hbar)^j;
\end{equation}
\eq{delta-h} \Delta\colon H(u)\mapsto\sum\limits_{j=0}^\infty(-1)^j (j+1)F(u+\hbar)^jH(u)\otimes H(u)E(u+\hbar)^j.
\end{equation}

\prop{sl2-coprod}
Let $l,k\le 0$ and $m = l +k$. Then the coproduct $\Delta\colon \bY_m(\fsl_2)\to \bY_{l}(\fsl_2)\otimes \bY_{k}(\fsl_2)$ is also given by the formulas  \refe{delta-e}--\refe{delta-h}, where by abuse of notation $E(u), F(u), H(u)$ denote the generating series for each respective algebra.
\eprop
\prf
We make use of the following commutative diagram:
\eq{eq:sl2-coprod}
\begin{CD}
\bY_0 @>{\Delta}>> \bY_0\otimes \bY_0 \\
@VV{\iota_{0,l,k}}V
@VV{\iota_{0,l,0}\otimes \iota_{0,0,k}}V\\
\bY_m @>{ \Delta}>> \bY_l \otimes \bY_k
\end{CD}
\end{equation}
from the statement of \reft{generalcoproduct}.  By commutativity, we may compute the coproduct of any elements in the image of $\iota_{0,l,k}\colon \bY_0 \rightarrow \bY_m$ by passing around the top of the diagram.  Modulo accounting for the shift homomorphisms involved, this is given by Molev's formulas \refe{delta-e}--\refe{delta-h}.

Note that the homomorphism $\iota_{0,l,k}\colon \bY_0 \rightarrow \bY_m$ is not surjective: the generators  $E^{(r)}, F^{(s)}$ are not in its image for $1\leq r\leq l$ and $1\leq s\leq k$.  However, the coproducts of these elements were explicitly described in \refss{Definition of coproduct}.  Piecing these coproducts together with those computed above, the claim follows.
\epr

\cor{quantized multiplication}
In this case, the comultiplication $ \Delta\colon \bY_m(\fsl_2)\to \bY_{l}(\fsl_2)\otimes \bY_{k}(\fsl_2)$ quantizes the multiplication map $ \CW_k \times \CW_l \rightarrow \CW_m $.  (In other words, in the notation of \refss{comult}, we have $ \Delta_{k,l}^1 = \Delta_{k,l}^2$.)
\ecor
\prf
For the ordinary Yangian $Y(\fsl_2)$, the classical limit of the formulas \refe{delta-e}-- \refe{delta-h} corresponds to multiplication in the group $(PGL_2)_1[[z^{-1}]]$, written with respect to Gauss decompositions.  Explicitly, any element of $(PGL_2)_1[[z^{-1}]]$ can be written uniquely in the form
\eq{gauss mult 1} g=  \begin{pmatrix} 1 & 0 \\ e & 1 \end{pmatrix} \begin{pmatrix} 1 & 0 \\ 0 & h \end{pmatrix} \begin{pmatrix} 1 & f \\ 0 & 1 \end{pmatrix}
\end{equation}
with $e, f\in z^{-1} \BC[[z^{-1}]]$ and $h \in 1 + z^{-1} \BC[[z^{-1}]]$.  The product of two such elements, rewritten in the above form, is
\eq{gauss mult 2}g_1 g_2 = \begin{pmatrix} 1 & 0 \\
e_1 + \frac{h_1 e_2}{1+f_1 e_2} & 1 \end{pmatrix} \begin{pmatrix} 1 & 0 \\
0 & \frac{h_1 h_2}{(1+f_1 e_2)^2} \end{pmatrix} \begin{pmatrix} 1 &
\frac{f_1 h_2}{1+ f_1 e_2} + f_2 \\ 0 & 1 \end{pmatrix}
\end{equation}
On the level of coordinate rings, this corresponds precisely to \refe{delta-e}--\refe{delta-h} with $\hbar = 0$.

Any $g\in \CW_n$ can also be written uniquely in the form \refe{gauss mult 1}, but where now we take $h\in z^n+ z^{n-1} \BC[[z^{-1}]]$.  When $k, l\leq 0$ the multiplication map $\CW_k \times \CW_l \rightarrow \CW_m$ is given by matrix multiplication, and \refe{gauss mult 2} generalizes immediately.  This proves the claim.
\epr

\rem{general-coproduct antidominant}
Assuming one has explicit formulas for the coproduct $\Delta\colon Y_0(\fg) \rightarrow Y_0(\fg)\otimes Y_0(\fg)$, a similar logic to \refp{sl2-coprod} gives explicit formulas for the coproduct $\Delta\colon Y_\mu(\fg) \rightarrow Y_{\mu_1}(\fg) \otimes Y_{\mu_2} (\fg)$ in the case when $\mu,\mu_1,\mu_2$ are all antidominant.
\erem

\ssec{ysl2mu}{Shifted Yangian of $\fsl_2$ and Toda}

According to~\cite[Theorem~B.18]{bfn16}, for a simple simply laced $\fg$,
there is a homomorphism from the shifted Yangian $\bY_{\mu}(\fg)$ to a quantized
Coulomb branch. 
Let us describe it in the simplest case $\fg=\fsl_2$, $\mu=-2n$, $\lambda=0$ where $n$ is a positive integer.

\prop{y2hgr}(\cite[Theorem~B.18]{bfn16}) There is a homomorphism
$\overline{\Phi}{}_{-2n}^{0}\colon \bY_{-2n}(\fsl_2)\to H_\bullet^{G^\vee_\CO\rtimes\BC^\times}(\Gr_{G^\vee})$ for $G=G^\vee=GL_n$. We have
\begin{gather*}
\overline{\Phi}{}_{-2n}^{0}(A^{(p)})= e_p \in
H^\bullet_{G^\vee_\CO\rtimes\BC^\times}(pt)\subset
H_\bullet^{G^\vee_\CO\rtimes\BC^\times}(\Gr_{G^\vee}), \\
\overline{\Phi}{}_{-2n}^{0}(F^{(1)})=[\Gr_{G^\vee}^{\varpi_1}], \quad
\overline{\Phi}{}_{-2n}^{0}(E^{(1)})=(-1)^n[\Gr_{G^\vee}^{-\varpi_1}]
\end{gather*}
\eprop

Note that according to \refl{lem-cl-gen} and the paragraph following it,
the homomorphism $\overline{\Phi}{}_{-2n}^{0}$ is surjective.

The ring $H^\bullet_{G^\vee_\CO\rtimes\BC^\times}(pt)$ gets identified with the center of the universal enveloping algebra $ZU_\hbar(\fg)$ via the Satake correspondence. The Harish-Chandra homomorphism identifies the center $ZU_\hbar(\fg)$ with the algebra $\BC[\fh^*]^W$ of $W$-invariant polynomials with respect to the $W$-action shifted by $-\hbar\rho$.  Here we write $ e_p $ for the $p$-th elementary symmetric function in $\BC[\fh^*]^W$  (shifted by $-\hbar\rho_n$ where $\rho_n:=(\frac{n-1}{2},\frac{n-3}{2},\ldots,\frac{-n+1}{2})$).  So we can compute the images of $A^{(1)}$ and $A^{(2)}$ as elements of the center of the universal enveloping algebra $ZU_\hbar(\fg)$. Combining it with $\bbeta\colon H_\bullet^{G^\vee_\CO\rtimes\BC^\times}(\Gr_{G^\vee})\iso\CT_\hbar^n$ we get

\prop{y2tn} There is a surjective homomorphism
$\bbeta\circ\overline{\Phi}{}_{-2n}^{0}\colon \bY_{-2n}(\fsl_2)\to \CT_\hbar^n$ which takes the subalgebra generated by the $ A^{(p)} $ to $ ZU_\hbar(\fg) $.  In particular, we can see that the homomorphism takes $A^{(1)}$ to $C_1$, $A^{(2)}$ to $C_2-(\rho_n,\rho_n)\hbar^2$, $E^{(1)}$ to $-\varDelta$ and $F^{(1)}$ to $\varDelta'$.
\eprop

\rem{GKL}
In section 6 of \cite{GKL}, the authors construct certain elements $ A_n(\lambda), B_n(\lambda), C_n(\lambda) \in \CT_\hbar^n[[\lambda]] $.  They observe that these elements satisfy some (but not all) of the relations of the $ \fsl_2 $ Yangian.

It is easy to see that the elements $ A_n(\lambda), B_n(\lambda), C_n(\lambda) $ defined in \cite{GKL} are the images of the same named elements of $ \bY_{-2n}(\fsl_2)[[\lambda]] $ under the homomorphism $\bbeta\circ\overline{\Phi}{}_{-2n}^{0}$.  This explains why these elements satisfy the relations from \cite[(6.5)]{GKL}.

Moreover, the formulas (6.7) from \cite{GKL} are special cases of the formulas from \cite[Corollary~B.17]{bfn16}.
\erem

\ssec{quantum}{Compatibility of the coproducts}
According to \refp{y2tn} there is a homomorphism $\bbeta\circ\overline{\Phi}{}_{-2n}^{0}\colon \bY_{-2n}(\fsl_2)\to \CT_\hbar^n$. Twisting by the additive shift automorphisms
$T_\varepsilon$ (notations of~\refss{u-eps}) gives a family of homomorphisms $\bbeta\circ\overline{\Phi}{}_{-2n}^{0}[\varepsilon]:=\bbeta\circ\overline{\Phi}{}_{-2n}^{0}\circ T_\varepsilon\colon \bY_{-2n}(\fsl_2)\to \CT_\hbar^n$.

We have the following quantum version of \reft{classi}:

\th{quantu}
The following diagram commutes:
$$\begin{CD}
\bY_{-2k-2l}(\fsl_2) @>{\Delta}>> \bY_{-2k}(\fsl_2)\otimes \bY_{-2l}(\fsl_2)\\
@VV{\bbeta\circ\overline{\Phi}{}_{-2k-2l}^{0}}V
@VV{\bbeta\circ\overline{\Phi}{}_{-2k}^{0}[\frac{l\hbar}{2}]\otimes\bbeta\circ\overline{\Phi}{}_{-2l}^{0}[-\frac{k\hbar}{2}]}V\\
\CT_\hbar^{k+l} @>{\tau_{k,l}}>> \CT_\hbar^{k}\otimes \CT_\hbar^{l}
\end{CD}.$$
\eth

\prf
Follows from \refp{25} and \refl{4generatorscoproduct}.
\epr

\bigskip

\footnotesize{
{\bf M.F.}: National Research University
Higher School of Economics, Russian Federation,\\
Department of Mathematics, 6 Usacheva st, Moscow 119048;\\
Skolkovo Institute of Science and Technology;\\
Institute for Information Transmission Problems of RAS;\\
{\tt fnklberg@gmail.com}}

\footnotesize{
{\bf J.K.}:
University of Toronto, Department of Mathematics;\\
Room 6290, 40 St. George Street, Toronto, ON, Canada M5S 2E4;\\
{\tt jkamnitz@gmail.com}}

\footnotesize{
{\bf K.P.}:
University of Toronto, Department of Mathematics;\\
Room 6290, 40 St. George Street, Toronto, ON, Canada M5S 2E4;\\
{\tt khoatd.pham@mail.utoronto.ca}}

\footnotesize{
{\bf L.R.}: National Research University
Higher School of Economics, Russian Federation,\\
Department of Mathematics, 6 Usacheva st, Moscow 119048;\\
Institute for Information Transmission Problems of RAS;\\
{\tt leo.rybnikov@gmail.com}}

\footnotesize{
{\bf A.W.}: Perimeter Institute for Theoretical Physics, \\
31 Caroline St. N, Waterloo, ON, Canada N2L 2Y5 \\
{\tt alex.weekes@gmail.com}}

\end{document}